\documentclass[11pt,a4paper]{article}
\usepackage{amsmath,amsfonts,amstext, amsthm, pgf, tikz, paralist, color, colortbl, array, graphicx, hyperref, fancyhdr,mathtools}
\author{Marie-Louise Bruner  \\ Vienna University of Technology, Austria \\ \texttt{marie-louise.bruner@tuwien.ac.at}
}
\title{On restricted permutations on regular multisets\thanks{This work was supported by the Austrian Science Foundation FWF, grant P25337-N23.}}
\date{\today}

\theoremstyle{plain}
\newtheorem{Thm} {Theorem} [section]
\newtheorem{Prop} [Thm] {Proposition}

\newtheorem{Def} [Thm]{Definition}
\newtheorem{Cor} [Thm] {Corollary}
\newtheorem{Cla} {Claim}
\theoremstyle{definition}
\newtheorem{Ex} [Thm]{Example}
\theoremstyle{remark}
\newtheorem{Rem} [Thm]{Remark}

\newcommand{\unt}{\underline}
\newcommand{\pe}{permutation}
\newcommand{\mul}{multiset}
\newcommand{\pes}{permutations }

\newcommand{\zeill}{\parbox[0pt][2em][c]{0cm}}
\newcommand{\zeil}{\parbox[0pt][3em][c]{0cm}}
\begin{document}
\thispagestyle{empty}
\maketitle

\begin{abstract}
The extension of pattern avoidance from ordinary permutations to those on multisets gave birth to several interesting enumerative results. 
We study permutations on regular multisets, i.e.,  multisets in which each element occurs the same number of times. 
For this case, we close a gap in the work of Heubach and Mansour \cite{heubach2006avoiding} and complete the study of permutations avoiding a pair of patterns of length three. 
In all studied cases, closed enumeration formul\ae\ are given and well-known sequences appear. 
We conclude this paper by some remarks on a generalization of the Stanley-Wilf conjecture to permutations on multisets and words.

\end{abstract}

\pagestyle{fancy}
\section{Introduction}

\subsection{Restricted permutations on multisets}
One possible origin of restricted permutations can be found in an exercise of Knuth in \cite{knuth1968art} in which he proved that permutations sortable using one stack are exactly those that avoid the pattern $231$.
The first systematic study of pattern avoiding permutations was then performed by Simion and Schmidt in \cite{simion1985restricted}. Since then, restricted permutations have become a field of intensive research as witnessed by B\'{o}na's \textit{Combinatorics of Permutations}~\cite{bona_combinatorics_2004} and Kitaev's recent monograph \textit{Patterns in Permutations and Words} \cite{kitaev2011patterns}.

\begin{Def}
A \pe\ $\sigma=\sigma_1 \sigma_2 \ldots \sigma_n$ of length $n \geq k$ is said to \textit{contain} another \pe\ $\pi=\pi_1 \pi_2 \ldots \pi_k$ as a \emph{pattern } if we can find $k$ entries $\sigma_{i_1}, \sigma_{i_2}, ..., \sigma_{i_k}$ with $i_1 < i_2 < ... < i_k$ such that $\sigma_{i_a} < \sigma_{i_b} \Leftrightarrow \pi_a < \pi_b$. In other words, \emph{$\sigma$ contains $\pi$} if we can find a subsequence of $\sigma$ that is order-isomorphic to $\pi$. If there is no such subsequence we say that $\sigma$ \emph{avoids the pattern $\pi$}.
\label{def_avoidance_pes}
\end{Def}

\begin{Ex}
The \pe\ $\sigma=21543$ (written in one-line representation respectively as a sequence of integers) contains the pattern $\pi=132$, since the entries $143$ form a $132$-pattern. 
The pattern $\pi=123$ is however avoided by $\sigma$ since it does not contain any increasing subsequence of length three.
\end{Ex}

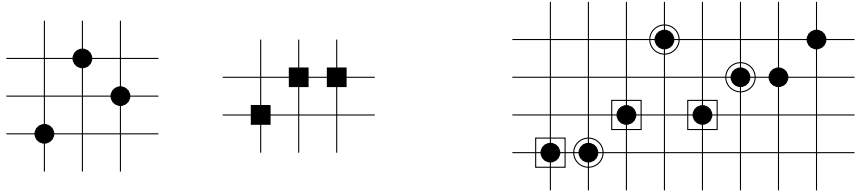
\begin{figure}
\begin{minipage}[hbt]{0.2\textwidth}
	\centering
	\begin{tikzpicture}
[0/.style={rectangle, draw, minimum size=28.5pt}, 
1/.style={circle, draw, fill=black, minimum size=7pt, inner sep=2pt}, 
m/.style={circle, draw, fill=none, minimum size=11pt}, 
2/.style={rectangle, fill=black!20, minimum size=14.25pt}, 
scale=0.5]

   
\foreach \y in {1,2,3}
\draw (0,\y)--(4, \y);
\foreach \x in {1,2,3}
\draw (\x,0)--(\x, 4);

\foreach \x/\y in { 1/1, 2/3, 3/2}
    \node[1] at (\x, \y) {};

\end{tikzpicture}
\end{minipage}
\hfill
\begin{minipage}[hbt]{0.2\textwidth}
	\centering
	\begin{tikzpicture}
[0/.style={rectangle, draw, minimum size=28.5pt}, 
1/.style={rectangle, draw, fill=black, minimum size=7pt, inner sep=1.5pt}, 
m/.style={circle, draw, fill=none, minimum size=11pt}, 
2/.style={rectangle, fill=black!20, minimum size=14.25pt}, 
scale=0.5]

   
\foreach \y in {1,2}
\draw (0,\y)--(4, \y);
\foreach \x in {1,2,3}
\draw (\x,0)--(\x, 3);

\foreach \x/\y in { 1/1, 2/2, 3/2}
    \node[1] at (\x, \y) {};

\end{tikzpicture}
\end{minipage}
\hfill
\begin{minipage}[hbt]{0.55\textwidth}
	\centering
	\begin{tikzpicture}
[0/.style={rectangle, draw, minimum size=28.5pt}, 
1/.style={circle, draw, fill=black, minimum size=7pt, inner sep=2pt}, 
m/.style={circle, draw, fill=none, minimum size=11pt}, 
m2/.style={rectangle, draw, fill=none, minimum size=11pt}, 
2/.style={rectangle, fill=black!20, minimum size=14.25pt}, 
scale=0.5]

   
\foreach \y in {1,2,3,4}
\draw (0,\y)--(9, \y);
\foreach \x in {1,2,3,4,5,6,7,8}
\draw (\x,0)--(\x, 5);

\foreach \x/\y in { 1/1, 2/1, 3/2, 4/4, 5/2, 6/3, 7/3 , 8/4}
    \node[1] at (\x, \y) {};
\foreach \x/\y in { 2/1,4/4, 6/3}
    \node[m] at (\x, \y) {};
\foreach \x/\y in { 1/1,3/2, 5/2}
    \node[m2] at (\x, \y) {};

\end{tikzpicture}
\end{minipage}
\caption{Pattern containment in permutations on multisets: The permutation to the right contains both patterns $132$ and $122$ as can be seen by the subsequences of elements marked by circles respectively squares.}
\label{pattern.avoidance.multi.grid}
\end{figure}

Considering \textit{\pe s on \mul s} (or words on a certain finite alphabet) is a natural generalization of ordinary \pe s. 
Intuitively, a \mul\ is a ``set'' in which elements may occur more than once.

\begin{Def}
A \emph{multiset} is a pair $(\mathcal{S},\mu)$, where $\mathcal{S}$ is a set and $\mu: \mathcal{S} \rightarrow \mathbb{N}\setminus \left\lbrace 0\right\rbrace$ is a function.
Then $\mu(i)$ denotes the number of times an element $i \in \mathcal{S}$ is repeated, it is called the \emph{multiplicity} of $i$.
For $\mathcal{S}=[n] \coloneqq \left\lbrace 1, 2, \ldots, n\right\rbrace$ we write $\left\lbrace 1^{\mu(1)}, 2^{\mu(2)}, \ldots, n^{\mu(n)}\right\rbrace$ instead of $(\mathcal{S}, \mu)$.
A \mul\ is called \emph{regular} in case $\mu(i)=m$ for all $i \in \mathcal{S}$ and some positive integer $m$.
We use the notation $[n]_m$ for $\left\lbrace 1^{m}, 2^{m}, \ldots, n^{m}\right\rbrace $.
A \emph{permutation} $\sigma$ on a multiset $\left\lbrace 1^{\mu(1)}, 2^{\mu(2)}, \ldots, n^{\mu(n)}\right\rbrace$ is then a sequence of length $\sum_{i \in \mathcal{S}}{\mu(i)}$ in which \hyphenation{every} every element $i \in [n]$ appears exactly $\mu(i)$ times.
The set of all permutations on $[n]_m$ is denoted by $\mathcal{S}_{n,m}$.
The \emph{length} of a permutation on $(\mathcal{S}, \mu)$ is the length of the corresponding sequence, i.e., $l=\sum_{i \in \mathcal{S}}{\mu(i)}$.
\label{Def:multiset}
\end{Def}

The notion of pattern avoidance respectively containment can be extended to \pe s on \mul s in a straightforward way and Definition~\ref{def_avoidance_pes} can also be applied to permutations on multisets. %
For a certain pattern to be contained in a \pe , repetitions in the pattern have to be represented by repetitions in the \pe . 
\begin{Ex}
The \mul -\pe\ $\sigma =11242334$ as depicted in Figure~\ref{pattern.avoidance.multi.grid} contains the ordinary pattern $\pi =132$ and the multiset-pattern $\pi = 122$ but avoids $\pi = 221$. 
\end{Ex}

Given a set of patterns $\Pi$, we denote by $\mathcal{S}_\mu(\Pi)=\mathcal{S}_{\mu(1), \mu(2), \ldots , \mu(n)}(\Pi)$ the set of all permutations on the multiset $([n],\mu)$ avoiding any of the patterns in $\Pi$. 
The cardinality of $\mathcal{S}_\mu(\Pi)$ is then denoted by $s_\mu(\Pi)$.
For the special case of regular multisets $[n]_m$, we write $\mathcal{S}_{n,m}(\Pi)$ and $s_{n,m}(\Pi)$.

The study of such restricted permutations on multisets, generalizing the ordinary ones, started only a few years ago. 
We mention here the work of Savage and Wilf \cite{savage2006pattern}, Myers \cite{myers2007pattern}, Heubach and Mansour \cite{heubach2006avoiding}, Albert et al. \cite{albert2001permutations} and Kuba and Panholzer \cite{kuba2010enumeration}. 
Studying patterns in permutations on multisets is of particular interest because of the connection to patterns in other combinatorial objects such as compositions and partitions, see e.g.\ \cite{heubach2006avoiding}.

Let us mention that in parallel to restricted permutations on multisets, \textit{restricted words} have been studied over the last few years. 
As introduced in Definition~\ref{Def:multiset}, permutations on a multiset $\left\lbrace 1^{\mu(1)}, 2^{\mu(2)}, \ldots, n^{\mu(n)}\right\rbrace$ are sequences of length $l=\sum_{i \in \mathcal{S}}{\mu(i)}$ containing every element $i \in [n]$ exactly $\mu(i)$ times. 
If we drop this restriction, that is to say if we allow for sequences of length $l$ containing every element $i \in [n]$ any possible number of times, i.e., from $0$ up to $l$ times, we obtain \textit{words} of length $l$ over the alphabet $[n]$.
In the same way as for permutations on multisets, we can define pattern containment and avoidance for words.
The study of patterns in words has indeed become a topic of intensive research within the last years as witnessed by the monograph \textit{Combinatorics of compositions and words} by Heubach and Mansour \cite{heubach2009combinatorics}.
Every word can be seen as a permutation of a certain mutlisets and vice versa every permutation on a multiset is also a word.
Thus the objects studied within these two branches of pattern avoidance are the same, whereas the way of counting them differs.
In this paper the enumerative results do not concern patterns in words but restricted permutations on multisets.
However, the last section of this article treats a generalization of the Stanley-Wilf conjecture both to permutations on multisets and to words.

\subsection{Results}

In this article, we close some gaps in the mentioned work for the special case of \textit{regular multisets}, completing the enumeration of permutations on regular multisets avoiding any possible combination of two patterns of length three. 
In all studied cases (there are seven of them), we obtain exact enumeration formul\ae\ and well-known sequences such as (generalized) Catalan numbers appear. 
Note that this is in contrast to the previous work on permutations on general multisets avoiding some pattern of length three.
Both in \cite{albert2001permutations} and  \cite{heubach2006avoiding} where permutations on multisets avoiding ordinary respectively multiset-patterns of length three were studied, many of the cases did not lead to exact enumeration formul\ae\ but merely to recursive results.

\begin{table}[h]
\begin{center}
\newcolumntype{C}{ >{\centering\arraybackslash} m{0.255\textwidth} }
\newcolumntype{D}{ >{\centering\arraybackslash} m{0.24\textwidth} }
\newcolumntype{E}{ >{\centering\arraybackslash} m{0.42\textwidth} }
\footnotesize{
\begin{tabular}{ccc}
\hspace{-0.25cm}
\begin{tabular}[t]{|C|} \hline 
\zeill{} Pairs of ordinary patterns   \\ \hline \hline

\vspace{0.1cm}   $s_{n,m}(123, 132)$: \newline no explicit formula, recursion in Section 4.1 in \cite{albert2001permutations} \\ \hline

\vspace{0.1cm}   $s_{n,m}(123, 231)=$ $\binom{nm}{m}+\binom{n-1}{2}m^2$  Section 4.2 in \cite{albert2001permutations}  \\ \hline

\vspace{0.1cm}  $s_{n,m}(123, 321)=\begin{cases} (m+1)^2 c_m &  n=3, \\ 	2(m+1) c_m &  n=4, \\ 0 &  n \geq 5.\end{cases}$  Section 4.3 in  \cite{albert2001permutations} \\ \hline

\vspace{0.1cm} $s_{n,m}(132, 213):$ \newline no explicit formula, recursion in Section 4.4 in  \cite{albert2001permutations} \\ \hline

\vspace{0.1cm} $s_{n,m}(132, 231)  =  c_m \cdot (m+1)^{n-2}$  Section 4.5 in  \cite{albert2001permutations} \\ \hline

\vspace{0.1cm}  $s_{n,m}(132, 312)=$   $\sum\limits_{i=1}^{n-1}{\binom{mn}{mi}}-\sum\limits_{i=1}^{n-2}{\binom{mn-m}{mi}}$ Section 4.6 in  \cite{albert2001permutations} \\ \hline 
\end{tabular}
 \hspace{-0.3cm}
&
\begin{tabular}[t]{|D|} \hline 
\zeill{} Pairs of multiset-patterns  \\ \hline \hline

\zeil{} $s_{n,m}(212, 221)=1$ \newline Theorem 5.5 in \cite{heubach2006avoiding} \\ \hline

\zeil{} $s_{n,m}(212, 121)=n!$ \newline Theorem 5.10 in \cite{heubach2006avoiding} \\ \hline

\zeil{} \cellcolor{black!10}$s_{n,m}(122, 112)=$ \vspace{-0.2cm}\newline  $\begin{cases} 2^{n-1}, & m \geq 3,\\ c_n, & m=2.\end{cases}$ \vspace{0.1cm}\newline Theorems \ref{Thm.m>2.112,122} and \ref{Thm.112,122=cn} \\ \hline

\zeil{} $s_{n,m}(122, 121)=c_n$ \newline Theorem 5.9 in \cite{heubach2006avoiding}\\ \hline

\zeil{} $s_{n,m}(122, 211)=0$ \newline Theorem 5.11 in \cite{heubach2006avoiding} \\ \hline

\vspace{0.2cm} $s_{n,m}(122, 221)= $ \newline  $ \begin{cases} 1 &  m=2, \\ 	0 & m \geq 3.				\end{cases}$ \vspace{0.1cm} \newline  Theorem 5.7 in \cite{heubach2006avoiding} \\ \hline
\end{tabular}
&
\hspace{-0.3cm}
\begin{tabular}[t]{|E|} \hline 
\zeill{} Pairs of one ordinary and one multiset-pattern  \\ \hline \hline

\vspace{0.11cm} 
$s_{n,m}(212, 123)=$ $\sum\limits_{j=0}^{n}{\frac{\binom{n}{j}\binom{n+ (m-1)j-1}{n-j}}{n+1-j}}$  \newline Theorem 9 in \cite{kuba2010enumeration}.\\ \hline

\vspace{0.1cm} $s_{n,m}(212, 132)= c_{n,m},$ \vspace{0.1cm} \newline Theorem 7 in \cite{kuba2010enumeration}.\\ \hline

\vspace{0.12cm}  \cellcolor{black!10} $s_{n,m}(122, 123)=s_{n,m}(122, 132)=c_{n,m},$  Theorem \ref{Thm.gen.catalan}\\ \hline

\vspace{0.11cm}  \cellcolor{black!10} $s_{n,m}(211, 213)=$   $m s_{n-1,m}(122, 231) + s_{n-2,m}(122, 231)$ \vspace{-0.2cm} \newline Theorem \ref{Thm.211,213}, explicit formula in Proposition \ref{Thm.explicit.211,213} \\ \hline

\vspace{0.1cm} \cellcolor{black!10} $s_{n,m}(122, 213)=$ \newline  $ 2 s_{n-1,m}(122, 213) + s_{n-2,m}(122, 213)$ \vspace{0.175cm} \newline Theorem \ref{Thm.122,213}, explicit formula in Proposition \ref{coeff122,213}\\ \hline

\vspace{0.11cm} \cellcolor{black!10} $s_{n,m}(122, 312)= \left(n-1\right)\cdot m +1$ \newline Theorem \ref{Thm.122,312} \\ \hline

\vspace{0.2cm} \cellcolor{black!10}  $s_{n,m}(122, 321)= \begin{cases} 1 &  n=1, \\ 	m+1 &  n=2, \\ 0 &  n \geq 3.				\end{cases}$ \newline Theorem \ref{Thm.122,321} \\ \hline
\end{tabular} \\ 

\end{tabular}
}
\end{center} 
\caption{Enumerative results concerning permu\-tations on regular multisets avoiding any of the possible pairs of patterns of length three. The results marked in gray are new and are proved in this work.}
\label{tab:all.results} 
\end{table}

The enumerative results for permutations on regular multisets avoiding a pair of patterns of length three are summarized in Table \ref{tab:all.results} in which our new results have been marked in gray. 
The cases presented here are representatives of Wilf-equivalence classes of pattern-restricted permutations. 
By this, the following is meant: two sets of patterns $\Pi_1$ and $\Pi_2$ are said to be \textit{Wilf-equivalent} if $s_{n,m}(\Pi_1)=s_{n,m}(\Pi_2)$ holds for all $m$ and $n$. 

Wilf-equivalence can be obtained for symmetry reasons for many pairs of patterns: For a given permutation $\sigma$ (on an ordinary set or on a multiset $([n],\mu)$), we call the permutation $\sigma^r$ that is obtained by reading $\sigma$ from right to left its \textit{reverse} and the permutation $\sigma^c$ that is obtained by replacing $i$ by $n-i+1$ its \textit{complement}. 
If a permutation $\sigma$ avoids a certain pattern $\pi$, then its reverse $\sigma^r$ avoids the pattern $\pi^r$ and its complement $\sigma^c$ avoids the pattern $\pi^c$. 
This holds not only for single permutations, but also for sets of permutations and for \pes on regular\footnote{Note that $\sigma^c$ is a \pe\ on the regular multiset $[n]_m$ iff $\sigma$ is a \pe\ on $[n]_m$. This does not hold for multisets in general.} multisets. Therefore $\pi$ and $\pi^c$ as well as $\pi^r$ are Wilf-equivalent for all patterns $\pi$. 
This reduces the number of pairs that have to be considered from 66 to 20.
For pairs of patterns consisting of one ordinary and one multiset-pattern, we obtain the following equivalence classes: 
$\left\lbrace (212, 123)\right\rbrace $ and $\left\lbrace (212, 132)\right\rbrace $ (already studied by Kuba and Panholzer), $\left\lbrace (122, 123)\right\rbrace $, $\left\lbrace (122, 132)\right\rbrace $, $\left\lbrace (211, 213)\right\rbrace $, $\left\lbrace (122, 213)\right\rbrace $, $\left\lbrace (122, 312)\right\rbrace $ and $\left\lbrace (122, 321)\right\rbrace $.

In all cases, it is assumed that $m \geq 2$ since $m=1$ leads to the case of ordinary permutations which has been treated in \cite{simion1985restricted}. 
In this table, $c_n=\frac{1}{n+1}\binom{2n}{n}$ denotes the $n$-th Catalan number and $c_{n,m}=\frac{1}{m\cdot n+1}\binom {mn+n}{n}$ is the $(n,m)$-th generalized Catalan number that will be introduced in Section \ref{sec.122,123}. 
Note that Catalan numbers already occurred in the context of ordinary  restricted permutations.
Also note that the case that one of the forbidden patterns is $111$ is omitted in the table, since this can be treated very easily.
On the one hand, if $m=2$, all permutations avoid $111$ and thus $s_{n,2}(111,\pi)=c_n$ for all ordinary patterns $\pi$ of length three and all $n\geq 0$.
On the other hand, if $m>2$, all permutations contain $111$ and thus $s_{n,m}(111,\pi)=0$ for all ordinary patterns $\pi$ of length three and all $n\geq 0$.

\subsection{Organization}

In the work of Heubach and Mansour \cite{heubach2006avoiding} the case of $(122,112)$-avoiding permutations remained open.
We start by closing this gap for the special case of regular multisets.
This result is followed by the Wilf-equivalence of the two pairs of patterns $(122,123)$ and $(122,132)$. 
The subsequent sections present enumeration formul\ae\ for the pairs of patterns $(122,123)$, $(122,213)$, $(122,231)$, $(122,312)$ as well as $(122,321)$. 
We conclude this paper by stating the fact that a generalization of the former Stanley-Wilf conjecture also holds for permutations on multisets.
This section concerns permutations on general multisets as well as words.

\subsection{Methods}

When proving enumeration formul\ae\ for $[n]_m$-\pes avoiding certain patterns, we will often use an enumeration technique called \textit{generating trees} that was introduced in the study of Baxter \pes and has been systematized by Julian West in \cite{west1996generating} in the context of pattern avoidance.

A family of combinatorial objects can often be identified with its generating tree, a rooted, labelled tree that captures the recursive structure of the studied family.
In this tree, the nodes correspond to the studied objects and the height of the nodes, i.e., the distance from the root node, corresponds to their size. 

To be more precise, a \textit{generating tree} is a rooted, labelled tree having the property that the labels of the children of any given node $i$ can be determined by the label of $i$ itself. 
Therefore every generating tree can uniquely be described by a recursive definition, the so-called \textit{rewriting rule}, consisting of:
\begin{itemize}
\item the label of its root (this corresponds to the basis of an induction),
\item a set of succession rules, explaining how to derive the number of children and their labels for any given parent node when the label of the parent is given (this corresponds to the induction step).
\end{itemize}

These generating trees then lead to equations for (bivariate) generating functions that can often be solved using the so-called \textit{kernel method}. This method was first used as a trick to solve functional equations; one possible origin can be found in an exercise of Knuth \cite{knuth1968art} (as a matter of fact, this exercise also dealt with stack-sortable permutations and is followed by the one showing that these are exactly the 231-avoiding permutations). It is mainly thanks to Bousquet-M\'{e}lou, Flajolet, Petkov\v{s}ek et al., see \cite{banderier2002generating} and \cite{DBLP:journals/dm/Bousquet-MelouP00}, that the \textit{kernel trick} was placed on solid grounds and may now be used as a method for solving various functional equations. This method was amongst many others used by Bousquet-M\'{e}lou who attacked and solved further enumeration problems for pattern restricted permutations such as vexillary permutations and Baxter permutations \cite{DBLP:journals/combinatorics/Bousquet-Melou02}.
For details and an application of the kernel method, see the proof of Theorem~\ref{Thm.112,122=cn}.

Besides generating trees, generating functions and the kernel method, we will also use bijective proofs that will provide supplementary insight for two of the studied pairs of patterns.

\section{Avoiding the patterns $122$ and $112$}
\label{section112,122}

For the case of regular multisets,  we close the gap in the work of Heubach and Mansour \cite{heubach2006avoiding}, deducing enumeration formulae for \pes on multisets avoiding both the pattern $112$ and the pattern $122$. We differ from  their work by restricting ourselves to regular multisets, i.e., multisets where every element occurs the same number of times. 

\begin{Thm}
For $m \geq 3$, $s_{n,m}(112,122)=2^{n-1}$.
\label{Thm.m>2.112,122}
\end{Thm}

\begin{proof} 
By induction over $n$.
For $n=1$ it is clear that $\mathcal{S}_{1,m}(112,122)=\left\lbrace 11\ldots 1\right\rbrace $.

Let $\sigma$ be a permutation in $\mathcal{S}_{n,m}$ which avoids $112$ and $122$. The element $\left(n+1\right)$ should be inserted $m$ times - in correct positions - in order to produce an element of $\mathcal{S}_{n+1,m}$ avoiding the given patterns. 
First observe that the permutation $\sigma$ has to have at least two $n$'s at its beginning, otherwise it will be of the form $\sigma_i \ldots n \ldots n \ldots$ or $n \sigma_j \ldots n \ldots n \ldots$ where $\sigma_i$ and $\sigma_j$ are smaller than $n$ and thus $(\sigma_i n n)$ respectively $(\sigma_j n n)$ forms a 122-pattern.

Let us insert the $m$ copies of the element $(n+1)$ one after another without producing the forbidden patterns.
When inserting the first $\left(n+1\right)$ there are two possibilities: it can be placed either at the beginning of the permutation or directly after the first $n$. Otherwise, i.e., if it were to be placed somewhere after the second $n$, the permutation would be of the form $n n \ldots (n+1) \ldots$ and contain a 112-pattern. 
For the second, third and any following $\left(n+1\right)$ there is no other choice than to place them at the beginning of the permutation, otherwise a $122$-pattern would be created. 
Therefore every element in $\mathcal{S}_{n,m}(112,122)$ leads to two elements in $\mathcal{S}_{n+1,m}(112,122)$.
Thus, using the induction hypothesis, one obtains $s_{n+1,m}(112,122)=2\cdot s_{n,m}(112,122)=2\cdot 2^{n-1}=2^{n}.$ 
\end{proof}

We now deal with the case $m=2$.
The idea of the proof of  Theorem \ref{Thm.112,122=cn} is to describe the $(112, 122)$-avoiding \pes by means of their generating tree. This   will then lead to an equation for their generating function that can be solved with the help of the kernel method. In addition to this proof, we provide a bijection from $(112,122)$-avoiding \pes to Dyck words, giving further insight into this result.

\begin{Thm}
For $m=2$, $s_{n,m}(112,122)=c_n$, where $c_n$ denotes the $n$-th Catalan number.
\label{Thm.112,122=cn}
\end{Thm}

\begin{proof} 
We first describe the generating tree of $(112,122)$-avoiding permutations and then derive the coefficients of its generating function using the \textit{kernel method}.

Generating tree: Starting with a permutation $\sigma \in S_{n,m}$ which avoids $112$ and $122$, we determine where the two new elements $\left(n+1\right)$ may be inserted without producing the forbidden patterns.

Let us start by inserting one copy of $(n+1)$: 
This element may only be inserted before the first repetition of any element of $[n]$.
Otherwise, the new permutation would be of the form $\ldots \sigma_i \ldots \sigma_i \ldots (n+1) \ldots$ and thus contain a $112$-pattern. 
Let $r_\sigma$ denote the position of the first repetition in the permutation $\sigma$, i.e., $r_\sigma=\text{min}\left\{r \in [2n] | \exists \text{ } i<r: \sigma_i=\sigma_r\right\}$. 
For the empty permutation $\epsilon$ we set $r_\epsilon \coloneqq 1$. 
Now, if $r_\sigma=i$, the first $\left(n+1\right)$ can be placed in front of any of the first $i$ elements and thereby yield $i$ different permutations. 
Now that one copy of $(n+1)$ has been placed, there is only one possibility for the second copy.
The second $\left(n+1\right)$ may only be placed at the beginning of the permutation, otherwise a 122-pattern is created. 

For example, if $\sigma=2121$, $r_\sigma=3$ and we have three choices for placing the first $3$, namely $32121$, $23121$ and $21321$. 
The second $3$ must be placed at the beginning, so $\sigma$ gives us three permutations in $\mathcal{S}_{3,2}$ which avoid both patterns $112$ and $122$: $\sigma'=332121$, $\sigma''=323121$ and $\sigma'''=321321$ with $r_{\sigma'}=2$, $r_{\sigma''}=3$ and $r_{\sigma'''}=4$.

When placing two new elements in a permutation $\sigma$ with $r_\sigma=i$, we obtain one permutation $\tilde{\sigma}$ each with $r_{\tilde{\sigma}}=j$ for all $j \in \left\{2, \ldots, i+1 \right\}$. We can sum up these results in the generating tree of $(112,122)$-avoiding permutations. Its nodes are labelled by the position of the first repetition, i.e., with $r_\sigma$, as shown in Figure~\ref{perm}. 

\begin{figure}
\begin{minipage}[t]{\textwidth}
\centering
\vspace{0pt}
\begin{tikzpicture}[scale=1.2]
\tikzstyle{every node}=[font=\tiny] 
\tikzstyle{level 1}=[level distance=1cm, sibling distance=2.25cm] 
\tikzstyle{level 3}=[level distance=1cm, sibling distance=0.53cm] 
\tikzstyle{level 4}=[level distance=1.5cm, sibling distance=0.68cm] 
\tikzstyle{level 5}=[level distance=0.3cm]

\node (root) [font=\small]{$\epsilon$} child { node {1\unt{1}} child { node  {2\unt{2}11} 
child {
	node (1){3\unt{3}2211}
	child {
		node {4\unt{4}332211}
	}
	child[draw=white] {
		node {}
		child[draw=white] {node (2){43\unt{4}32211}}
	}
}
child[missing]
child [missing]
child {
	node (3) {32\unt{3}211}
	child {
		node {4\unt{4}323211}
	}
	child [draw=white] {
		node {}
		child[draw=white] { node (4) {43\unt{4}23211}}
	}
	child {
		node {432\unt{4}3211}
	}
}
}
child [missing]
child {
node  {21\unt{2}1}
child {
	node (5){3\unt{3}2121}
	child [draw=white]{
		node {}
		child[draw=white] { node (6){4\unt{4}332121}}
	}
	child {
		node {43\unt{4}32121}
	}
}
child [missing]
child [missing]
child [missing]
child {
	node (7){32\unt{3}121}
	child {
		node {4\unt{4}323121}
	}
	child [draw=white] {
		node {}
		child [draw=white] {node (8){43\unt{4}23121}}
	}
	child {
		node {432\unt{4}3121}
	}
}
child [missing]
child [missing]
child [missing]
child {
	node (9) {321\unt{3}21}
	child[draw=white] {
		node {}
		child [draw=white] {node (10) {4\unt{4}321321}}
	}
	child {
			node {\color{white}-- \color{black}43\unt{4}21321}
	}
	child[draw=white] {
		node {}
		child [draw=white] {node (11){\textcolor{white}{456}432\unt{4}1321}}
	}
	child {
			node {4321\unt{4}321}
	}
}
}
}
;
\draw (1) -- (2) ;
\draw (3) -- (4) ;
\draw (5) -- (6) ;
\draw (7) -- (8) ;
\draw (9) -- (10) ;
\draw (9) -- (11) ;
 
\end{tikzpicture} 
\end{minipage}
\begin{minipage}[t]{\textwidth}
\vspace{0pt}\raggedright
\centering
\vspace{1cm}
\begin{tikzpicture}[scale=1.4]
[font=\footnotesize]
\tikzstyle{every node}=[rectangle, draw=black, text centered] 
\tikzstyle{level 1}=[level distance=1cm] 
\tikzstyle{level 2}=[ sibling distance=4cm, level distance=1cm] 
\tikzstyle{level 3}=[ sibling distance=1.75cm, level distance=1cm]
\tikzstyle{level 4}=[ sibling distance=0.48cm, level distance=1.5cm]
\node (root)  {1} child { node  {2} child { node  {2} 
child {
	node {2}
	child {
		node {2}
	}
	child {
		node {3}
	}
}
child {
	node {3}
	child {
		node {2}
	}
	child {
		node {3}
	}
	child {
		node  {4}
	}
}
}
child {
node  {3}
child {
	node {2}
	child {
		node {2}
	}
	child {
		node {3}
	}
}
child {
	node {3}
	child {
		node {2}
	}
	child {
		node {3}
	}
	child {
		node {4}
	}
}
child {
	node {4}
	child {
		node {2}
	}
	child {
			node {3}
	}
	child {
			node {4}
	}
	child {
			node {5}
	}
}
}
}
; 
\end{tikzpicture} 
\vspace{0.5cm}
\end{minipage}
\caption{Generating tree of $(112,122)$-avoiding permutations on regular multisets with $m=2$. In the top tree the first repetition is underlined in every permutation; in the bottom tree the nodes are labelled by the position $r_\sigma$ of the first repetition.}
\label{perm}
\end{figure}

The generating tree of these permutations can be described by the simple rewriting rule:
\begin{align}
\notag  (1) & \\
(r) & \longrightarrow (2)(3)\ldots (r)(r+1) 
\label{rew.rule.112,122}
\end{align}
Note that this is the same generating tree as used by Bousquet-M\'elou in \cite{DBLP:journals/combinatorics/Bousquet-Melou02} for $123$-avoiding permutations, thus already implying that $(112,122)$-avoiding permutations are counted by the Catalan numbers.
We will however perform this proof anyway, in order to demonstrate the proof method used later on.

Generating function: Let 
\[
S(z,u) =	\sum_{\sigma \in \mathcal{S}_{n,2}(112,122)} {z^{s_\sigma}u^{r_\sigma}} = \sum_{n,k \geq 0}{s(n,k)z^n u^k}
\]
be the associated generating function, counting the nodes of the tree by their height and their label, respectively counting the $(112,122)$-avoiding permutations $\sigma$ by the size of their alphabet $s_\sigma$ and the position $r_\sigma$ of their first repetition. 

Using the rewriting rule (\ref{rew.rule.112,122}), one obtains:
\[
S(z,u)  =  u + zu^2 \sum_{n,k \geq 0}{s(n,k)z^n \frac{1-u^k}{1-u}} = u + zu^2 \frac{S(z,1)-S(z,u)}{1-u}.
\]
Rewriting this equation gives:
\begin{align}
K(z,u) \cdot S(z,u) = u\left(1-u+zuS(z,1)\right),
\label{eqn.kernel.112,122}
\end{align}
where $K(z,u)=zu^2-u+1$ is called the \textit{kernel} of the equation. This equation can be solved using the \textit{kernel method}. The roots of $K(z,u)$ are:
\[
u_1(z) = \frac{1 + \sqrt{1-4z}}{2z} \text{ and } u_2(z) = \frac{1 - \sqrt{1-4z}}{2z}.
\]
Note that $\lim_{z \to 0}u_1(z) \rightarrow \infty$ and $\lim_{z \to 0}u_2(z) \rightarrow 1$.
Therefore $u_2(z)$ is the only of the two roots of $K(z,u)$ that can be expanded into a power series in $z$ around $0$.
This root may be plugged into $S(z,u)$ since $S(z,u)$ is a series in $z$ with polynomial coefficients in $u$. 
We then obtain that the right-hand side of equation (\ref{eqn.kernel.112,122}) must vanish for $u=u_2(z)$. 
In particular, this means that $u_2(z) - 1 - z u_2(z) S(z,1)=0$, implying $S(z,1)=u_2(z)$.

It is well-known that this is the generating function of Catalan numbers:
\[
S(z,1) = \frac{1 - \sqrt{1-4z}}{2z} = \sum_{n \geq 1}{\frac{1}{n+1}\binom{2n}{n} z^n}.
\]
\end{proof}

We can also prove the result of Theorem \ref{Thm.112,122=cn} by giving a bijection from $\mathcal{S}_{n,2}(112,122)$ to the set of Dyck words of length $2n$. 

\begin{Def}
A \emph{Dyck word} of length $2n$ is a string consisting of $n$ X's and $n$ Y's such that no initial segment of the string has more Y's than X's. We denote by $\mathcal{D}_{n}$ the set of all Dyck words of length $2n$. 
\end{Def}

For example, the elements of  $\mathcal{D}_{3}$ are: XXXYYY, XXYXYY, XXYYXY, XYXXYY, XYXYXY. It is well-known that Dyck words are counted by the Catalan numbers. A famous proof of this result was given by Andr\'{e} \cite{andré1887solution}.

In the following, we construct a bijection from $\mathcal{S}_{n,2}(112,122)$ to $\mathcal{D}_n$.

Let us start with a Dyck word $w$ of length $2n$. The X's correspond to the first set of elements $1, \ldots, n$ and the Y's to the second set. Replace the sequence of X's by the decreasing sequence $n, n-1, \ldots, 2, 1$ and then do the same for the sequence of Y's. This leads to an element $\sigma$ of $\mathcal{S}_{n,2}$.

For example, given the Dyck word XYXXYXYY, one obtains the permutation $44323121$.

It remains to show that the corresponding permutation $\sigma$ indeed avoids the patterns 112 and 122 for every Dyck word $w$. 
The permutation $\sigma$ is formed by two decreasing subsequences $\left(n, n-1, \ldots, 2, 1\right)$ and therefore all elements to the left of the first occurrence of every $i \in [n]$ must be larger than $i$. 
This means that a 122-pattern is impossible. On the other hand, all elements to the right of the second occurrence of every $i \in [n]$ must be smaller than $i$. Thus a 112-pattern is also impossible.

Now let us construct a Dyck word $w$ from a given element $\sigma$ of $\mathcal{S}_{n,2}$ avoiding the patterns 112 and 122. This is easy: the first occurrences of all elements of $\left[n\right]$ are transformed into X's and the second ones into Y's.
The obtained word contains $n$ X's and $n$ Y's and clearly is a Dyck word.
Indeed, an initial segment with more Y's than X's could only be obtained if the \textit{second} occurrence of some element of  $\sigma$ appeared before the \textit{first} occurrence.

\begin{Rem}
We have found a closed formula for the number of \pes on \textit{regular multisets} avoiding the two patterns $112$ and $122$ simultaneously but the case of multisets in general still remains unsolved. 
The fact that we obtained three different enumeration formul\ae\ - one for $m=1$, one for $m=2$ and one for the case $m \geq 3$ - for the number of $(112,122)$-avoiding \pes indicates that it is difficult to find such a formula in the general case when using the methods presented here.
Indeed, if we try to proceed in the same way as in the proof of Theorem \ref{Thm.112,122=cn} and want to construct the generating tree of these permutations, we would have to derive succession rules that depend on the height of a node: the number of children (and their labels) of a node at the height $i$ depends on the multiplicity $\mu(i+1)$. 
\end{Rem}

\begin{Rem}
The case where all multiplicities are larger or equal to $3$ is easy to deal with since the observations made in the proof of Theorem \ref{Thm.m>2.112,122} are still valid.
This implies that the number of permutations on the multiset $\left\lbrace 1^{\mu(1)}, \ldots, n^{\mu(n)} \right\rbrace $ avoiding the patterns $112$ and $122$ is equal to $2^{n-1}$ in case $\mu_i \geq 3$ for all $i \in [n]$. 
\end{Rem}

\section{Wilf-equivalence of $\left\{122,123\right\}$ and $\left\{122,132\right\}$}
\label{sec.122,123.wilf.122,132}

\begin{Thm}
For all $n\in \mathbb{N}$ and all $m \in \mathbb{N}$ it holds that 
$s_{n,m}(122,123)=s_{n,m}(122,132).$
\label{Thm122,123=122,132}
\end{Thm}

Many proofs have been given for the Wilf-equivalence of the patterns $123$ and $132$ for ordinary permutations. A  famous one can be found in the first systematic study of \textit{Restricted Permutations} \cite{simion1985restricted} and uses a bijection between $\mathcal{S}_{n}(123)$ and $\mathcal{S}_{n}(132)$ that keeps all \textit{left-to-right-minima} fixed. Given a permutation $\sigma$ on a (not necessarily regular) multiset, we call $\sigma_i$ a \textit{left-to-right-minimum} if $\sigma_i\leq \sigma_j$ holds for all $j < i$. 

\begin{proof}[Proof of Theorem \ref{Thm122,123=122,132}] We use the same map as in the proof of Simion and Schmidt \cite{simion1985restricted} and show that it is a bijection from $\mathcal{S}_{n,m}(122,132)$ to $\mathcal{S}_{n,m}(122,123)$: 
\begin{itemize}
\item Given a permutation $\sigma \in \mathcal{S}_{n,m}(122,132)$, the map $f$ keeps all left-to-right-minima fixed.
All other elements are removed from $\sigma$ and then inserted from left to right in the free positions in decreasing order.
\item Given a permutation $\tau \in \mathcal{S}_{n,m}(122,123)$, the inverse map $g$ also keeps the left-to-right-minima fixed.
The other elements are removed from $\tau$ and placed in the following way: at each free position from left to right, place the smallest element not yet placed that is larger than the closest left-to-right-minimum to the left of the given position. 
\end{itemize}
Let us give a simple example. Consider the permutation $\sigma=43421231$ on the regular multiset $[4]_2$. 
This permutation clearly avoids the two forbidden patterns $122$ and $132$.
The left-to-right minima are $4,3,2,1, 1$ and removing all other elements leads to the permutation 43\_21\_\_1.
We obtain $f(\sigma)$ by inserting the removed elements $4, 2, 3$ in decreasing order: $f(\sigma)=43421321$.
Again it is easy to check that $f(\sigma)$ avoids both $122$ and $123$.
Observe that the $123$-pattern in $\sigma$ has been transformed into a $132$-pattern in $f(\sigma)$.

In the proof of Simion and Schmidt it was shown that for a given 132-avoiding permutation $\sigma$, $f(\sigma)$ is the only 123-avoiding permutation with the same set and positions of left-to-right-minima and conversely, for a given 123-avoiding permutation $\tau$, $g(\tau)$ is the only 132-avoiding permutation with the same set and positions of left-to-right-minima. It remains to show that for a $(132,122)$-avoiding permutation $\sigma$, its image $f(\sigma)$ is not only 123- but also 122-avoiding and reversely, that $g(\tau)$ is indeed 122-avoiding for a $(123,122)$-avoiding permutation $\tau$.

Note that for any 122-avoiding permutation on the regular multiset $[n]_m$ the following elements  will always appear in this order and are left-to-right-minima:
\[
\underbrace{n, \ldots, n}_{\left(m-1\right)\text{ times}},\underbrace{n-1, \ldots, n-1,}_{\left(m-1\right)\text{ times}}\ \ldots, \underbrace{2, \ldots, 2}_{\left(m-1\right)\text{ times}}, \underbrace{1, \ldots, 1}_{\left(m-1\right)\text{ times}}
\]
The other elements of the permutation lie in between the ones listed above, respecting the following rule: the $m$-th copy of the element $i$, $i \in [n]$, always lies to the right of the $(m-1)$ copies of $i$ listed above.
Depending on where they lie, there might be more than these left-to-right minima. 

This can easily be seen by induction over $n$. 
For $n=1$ the induction hypothesis is trivially true. Suppose you are given a 122-avoiding permutation $\sigma$ on $[n-1]_m$ and want to introduce  $m$ copies of the element $n$ in order to produce a 122-avoiding permutation on $[n]_m$. As seen earlier, one element $n$ may occur anywhere in $\sigma$ and all the remaining $n$'s have to be placed at the beginning of the permutation. This block of $\left(m-1\right)$ $n$'s will of course be a block of left-to-right-minima. The previous left-to-right-minima will stay left-to-right-minima since all elements of $\sigma$ are smaller than $n$ which proves the induction hypothesis. Since the maps $f$ and $g$ keep the left-to-right-minima fixed, this observation is also true for $f(\sigma)$ and $g(\tau)$, if $\sigma$ and $\tau$ are $122$-avoiding permutations. 

Now let $\sigma$ be a $(122,123)$-avoiding permutation and suppose $f(\sigma)$ contains a $122$-pattern formed by some entries $xyy$. 
One of the copies of the element $y$ that are involved in the 122-pattern must be a left-to-right-minimum by the remark made above.
This is a contradiction to the fact that the element $x$ lies to left of $y$ and is smaller than $y$.
Therefore $f(\sigma)$ avoids the pattern $122$.
The same argument holds for $g(\tau)$, where $\tau \in \mathcal{S}_{n,m}(122,123)$: $g(\tau)$ avoids the pattern $122$.
\end{proof}

\section{Avoiding the patterns 122 and 123}
\label{sec.122,123}

\begin{Thm}
For $m \geq 1$ and $n \in \mathbb{N}$ it holds that 
\[
s_{n,m}(122,123)=\frac{1}{m\cdot n+1}\binom {\left(m+1\right)\cdot n}{n}=c_{n,m}.
\]
\label{Thm.gen.catalan}
\end{Thm}
\begin{Rem}
The numbers $c_{n,m}=\frac{1}{m\cdot n+1}\binom {\left(m+1\right)\cdot n}{n}$ can be seen as one of many possible generalizations of the Catalan numbers. Indeed, $c_{n,1}=\frac{1}{ n+1}\binom {2 n}{n}$. These generalized Catalan numbers may also be seen as special cases of the so-called Rothe numbers  which are named after August Friedrich Rothe who was one of the first to investigate the properties of these sequences in \cite{rothe1793formulae}.
They are defined in the following way: $A_n(a,b)=\frac{a}{a+bn}\binom{a+bn}{n}$. It holds that  $A_n(1,m+1)=c_{n,m}$.
\label{rem.gen.cat}
\end{Rem}

\begin{proof}[Proof of Theorem \ref{Thm.gen.catalan}]
As in the proof of Theorem \ref{Thm.112,122=cn} we describe the generating tree of $(122,123)$-avoiding permutations and then derive the coefficients of its generating function using the kernel method.
Let $\sigma$ be a $[n]_m$-\pe\ avoiding the patterns 122 and 123, and let us insert $m$ copies of the element $(n+1)$ into $\sigma$ without producing one of the forbidden patterns. 
In order to avoid 122, $(m-1)$ occurrences of $\left(n+1\right)$ have to be inserted at the beginning of $\sigma$, and one $\left(n+1\right)$ may be placed anywhere. The $(m-1)$ copies of $\left(n+1\right)$ placed at the beginning will never be involved in a 123-pattern, so the only restriction we have is that the remaining $\left(n+1\right)$ may not be placed after an increasing subsequence of length two. 
In other words, if $a_\sigma$ is the position of the first ascent in $\sigma$, i.e., $a_\sigma= \text{min}\left\{i \in [n \cdot m]: \sigma_{i-1} < \sigma_{i}\right\}$, then the element $\left(n+1\right)$ may not be inserted after the $a_\sigma $-th position. 
For the empty \pe\ $\epsilon$ we set $a_\epsilon=1$ and for any \pe\ $\sigma$ with no ascents $a_\sigma=|\sigma|=n \cdot m+1$. 

One of the copies of $\left(n+1\right)$ may be inserted in front of any of the first $a_\sigma$ elements of $\sigma$ yielding $a_\sigma$ \pes $\tilde \sigma$ on $[n+1]_m$ that avoid both 122 and 123.
If this copy of $\left(n+1\right)$ is inserted at the beginning of $\sigma$, i.e., $\tilde \sigma$ starts with a block of $m$ $\left(n+1\right)$'s, no new ascents are created and $a_{\tilde \sigma}=a_{\sigma}+m$. 
If the position in $\sigma$ where the copy of $\left(n+1\right)$ is inserted is $j>1$, then $\sigma_{j-1}\left(n+1\right)$ always forms an ascent, and therefore $a_{\tilde \sigma}=m-1+j$.

The rewriting rule of the generating tree of $(122,123)$-avoiding permutations where the nodes are labelled by $a_\sigma$ is thus given by: 
\begin{align*}
\text{Root: }  (1) & \\
 (a) & \longrightarrow (m+a)(m+1)(m+2)\ldots (m+a-1) 
\end{align*}
Note the similarity to the rewriting rule of the generating tree of $122$- and $112$-avoiding \pes that are counted by the Catalan numbers (see the proof of Theorem \ref{Thm.112,122=cn}). 

For the associated generating function $S_m(z,u)=\sum_{n,a \geq 0}{s_{m}(n,a)z^nu^a}$ one obtains  with the help of the rewriting rule given above:
\begin{align*}
S_m(z,u)  & =  u + \sum_{\substack {n \geq 0 \\ a \geq 0}}{s_m(n,a)z^{n+1}\left(u^{a+m}+ u^{m+1} + \ldots + u^{a+m-1}\right)} \\
& =  u + z u^{m+1} \frac{S_m(z,1)-S_m(z,u)}{1-u},
\end{align*}
or equivalently: 
\begin{align}
K_m(z,u)S_m(z,u) = u(1-u+zu^mS_m(z,1)),
\label{gen.func.122,123}
\end{align}
where $K_m(z,u)=1-u+zu^{m+1}$ is the kernel of equation (\ref{gen.func.122,123}). 
This equation can again be solved using the \textit{kernel method}. 
As stated by Bousquet-M\'elou et al. in \cite{banderier2002generating}, this type of polynomial of degree $m+1$ has exactly one root $u_0(z)$ that can be expanded to a power series in $z$ around $0$, the other $m$ roots can be expanded to Laurent series in $z^{1/m}$ around $0$. %
This root may be plugged into $S_m(z,u)$, since $S_m(z,u)$ is a series in $z$ with polynomial coefficients in $u$. 
We then obtain that the right-hand side of equation \eqref{gen.func.122,123} must vanish for $u=u_0(z)$. 
In particular, this means that $1-u_0(z)+zu_0(z)^mS_m(z,1)=0$, implying $S_m(z,1)=u_0(z)$.

This root $u_0(z)$ may be developed into a power series using \textit{Lagrange's Inversion Formula}. 
From \cite{banderier2002generating} it is known that the constant term is $1$ and we can therefore write $u_0(z)=1+G$, where $G$ is a power series in $z$ with constant term $0$. Since $u_0(z)$ is a root of $1-u+zu^{m+1}$, we have $-G+z\left(1+G\right)^{m+1}=0$, implying
\[
G= z \phi (G), \text{where } \phi(G)=\left(1+G\right)^{m+1}.
\]
Note that $\phi(0)=1$ and $u_0(z)=1+G$ is a power series in $z$ and Lagrange's Inversion Formula can be applied. Noting that 
\[
\left(\phi(G)\right)^n=\left(1+G\right)^{\left(m+1\right)n}=\sum_{k \geq 0}{\binom{\left(m+1\right)n}{k}G^k} \text{ leads to}
\]
\begin{align*}
\left[z^0\right]1+G 	& =  \left[G^0\right]1+G=1, \\
\left[z^n\right]1+G 	& =  \frac1n \left[G^{n-1}\right]\left(\phi(G)\right)^n \\
						& =  \frac1n \binom{\left(m+1\right)n}{n-1} \text{ for } n \geq 1.
\end{align*}
Putting this together with the remarks made above, we conclude that
\[
\left[z^n\right]S_m(z,1)= \frac1n \binom{\left(m+1\right)n}{n-1} \text{ for } n \geq 1.
\]
To finish this proof, note that
\[
\frac1n \binom{\left(m+1\right)n}{n-1}=\frac{\left(\left(m+1\right)n\right)!}{n!\left(mn+1\right)!}=\frac{1}{mn+1}\binom{\left(m+1\right)n}{n}=c_{m,n}.
\]
\end{proof}

\begin{Cor}
For $m \geq 1$ and $n \in \mathbb{N}$ it holds that 
\[
s_{n,m}(122,132)=\frac{1}{m\cdot n+1}\binom {\left(m+1\right)\cdot n}{n}.
\]
\end{Cor}

With the help of the generating tree introduced in the proof above one can also construct a bijective proof of Theorem \ref{Thm.gen.catalan}. Indeed, one can show that $(122,123)$-avoiding multiset-\pes bijectively correspond to certain lattice paths.

\begin{Def}
For given integers $a, b $ and $n$ with $n \geq 0$, $a,b \geq 1$, $\mathcal{P}_n(a,b)$ denotes the set of all lattice paths from $(0,0)$ to $(a + bn, n)$ consisting of unit steps up -- $(0,1)$ -- and to the right -- $(1,0)$ -- not touching the line $\Delta: y = \frac{x-a}{b}$ except at the endpoint. 
\end{Def}

For an example of a lattice path in $\mathcal{P}_4(1,3)$, see the bottom picture of Figure~\ref{fig.14777}.

\begin{Rem}
It is obvious that Dyck words can also be interpreted as lattice paths from the origin $(0,0)$ to $(n,n)$ that lie above the line $y = x$ and may touch it, by translating a letter $X$ into a step up and a letter $Y$ into a step to the right and vice-versa. Such lattice paths are also called Dyck paths. By adding one step to the right at the end, we obtain lattice paths from $(0,0)$ to $(n+1,n)$ that do not touch the line $\Delta : y=x-1$ except at the end.   
Dyck words of length $2n$ can therefore  be bijectively identified with paths in $\mathcal{P}_n(1,1)$. Thus $|\mathcal{P}_n(1,1)|=|\mathcal{D}_n|=c_n=A_n(1,2)$. The following result generalizes this observation. For a proof see e.g. \cite{mohanty1979lattice}.
\end{Rem}

\begin{Thm}
For  integers $a, b $ and $n$ with $n \geq 0$, $a,b \geq 1$, it holds that
$|\mathcal{P}_n(a,b)|=A_n(a,b+1),$
where $A_n(a,b)$ is the generalized Catalan number introduced in Remark~\ref{rem.gen.cat}.
\end{Thm}

We are now going to show the following result:
\begin{Thm}
The elements of $\mathcal{P}_n(1,m)$ can bijectively be identified with \pes on the multiset $[n]_m$ that avoid the patterns $122$ and $123$ simultaneously. This implies 
\[
s_{n,m}(122,123)=|\mathcal{P}_n(1,m)|=A_n(1,m+1)=c_{n,m}.
\]
\end{Thm}

\begin{proof}
First we  bijectively identify a $(122,123)$-avoiding $[n]_m$-permu\-tation with a certain finite sequence of integers of length $n$. Then we do the same for all lattice paths in $\mathcal{P}_n(1,m)$.

Recall the definition of the generating tree of $(122,123)$-avoiding \pes made in the first proof of Theorem \ref{Thm.gen.catalan}. 
In this tree, every branch of length $n$ corresponds to a unique \pe\ on a multiset over the alphabet $[n]$ that avoids both mentioned patterns. 
The branch corresponding to a given \pe\ defines a sequence of length $(n+1)$, namely the sequence of the labels of its nodes. 
This sequence is well-defined and two different $n$-\pes cannot correspond to the same sequence of integers since, for any given node, each child has a different label. 
For an example, see the top part of Figure~\ref{fig.14777} where the sequence of integers corresponding to the permutation $\sigma=443322421311$ is represented. 

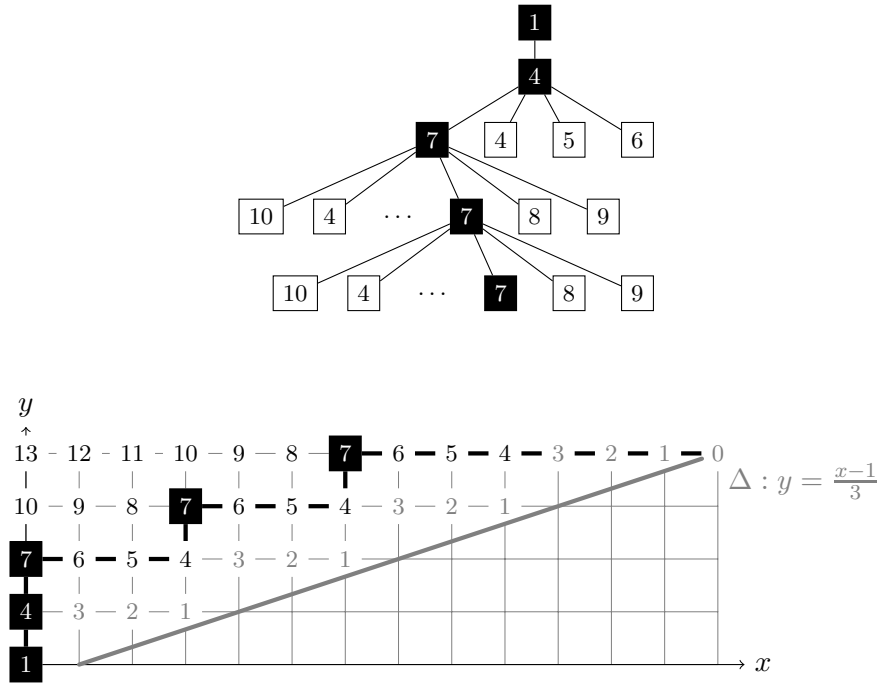
\begin{figure}[h]
\begin{minipage}{\textwidth}
\centering
\vspace{0pt}
\begin{tikzpicture}
[1/.style={rectangle,draw, fill=black, text centered, text=white, anchor=north},
2/.style={rectangle, draw, text centered, text=black, anchor=north},
font=\footnotesize,
level 1/.style={sibling distance=2cm, level distance=0.8cm},
level 2/.style={sibling distance=1.5cm, level distance=1cm},
level 3/.style={sibling distance=1.5cm, level distance=1.3cm},
level 4/.style={sibling distance=1.5cm, level distance=1.3cm},
scale=0.6]

\node (root) [1] {1} child { node [1]{4}
child { node [1]{7}
	child {node [2]{10}}
	child {node [2](a){4}}
	child [missing]
	child {node [1](b){7}
		child {node [2]{10}}
		child {node [2](i){4}}
		child [missing]
		child {node [1](j){7}}
		child {node [2]{8}}
		child {node [2]{9}}	
	}
	child {node [2]{8}}
	child {node [2]{9}}	
}
child { node [2] {4}
}
child { node [2]{5}
}
child { node [2]  {6}
}
}
;
\path (a) -- (b) node [midway] {$\ldots$};
\path (i) -- (j) node [midway] {$\ldots$};
\end{tikzpicture} 
\vspace{1cm}
\end{minipage}
\begin{minipage}{\textwidth}
\vspace{0pt}\raggedright
\centering
\begin{tikzpicture}[scale=0.7,cap=round]

  	\tikzstyle{axes}=[]
	\tikzstyle{important line}=[ultra thick, black]
	\tikzstyle{forbidden line}=[ultra thick, gray]
	\tikzstyle{1}=[circle, draw=red, fill=red, minimum size = 0.1cm]
	\tikzstyle{2}=[rectangle, draw=white, fill=white, minimum size = 0.01cm, font=\footnotesize]
	\tikzstyle{3}=[rectangle, draw=white, fill=white, minimum size = 0.001cm, text=gray, font=\footnotesize]
	\tikzstyle{4}=[rectangle, draw=black, fill=black, minimum size = 0.001cm, text=white, font=\footnotesize]

 \draw[style=help lines,step=1cm] (0,0) grid (13,4);
 
\begin{scope}[style=axes]
    \draw[->] (0,0) -- (13.5,0) node[right] {$x$};
    \draw[->] (0,0) -- (0,4.5) node[above] {$y$};
\end{scope}

\draw[-] (1,0) -- (13,4) node[below right] {$\textcolor{gray}{\Delta: y=\frac{x-1}{3}}$};

\draw[style=important line]
    (0,0)  -- (0,1);
\draw[style=important line]
    (0,1) -- (0,2);
\draw[style=important line]
    (0,2) -- (3,2);
\draw[style=important line]
    (3,2) -- (3,3);
\draw[style=important line]
    (3,3) -- (6,3);
\draw[style=important line]
    (6,3) -- (6,4);
\draw[style=important line]
    (6,4) -- (13,4);

\foreach \x in {4, 5, 6, 7, 8, 9, 10, 11, 12, 13}
    \node[2] at (13-\x, 4) {$\x$};
\foreach \x in {4, 5, 6, 7, 8, 9, 10}
    \node[2] at (10-\x, 3) {$\x$};
\foreach \x in {4, 5, 6, 7}
    \node[2] at (7-\x, 2) {$\x$};
\foreach \x in {4}
    \node[2] at (4-\x, 1) {$\x$};
\foreach \x in {1}
    \node[2] at (1-\x, 0) {$\x$};
\foreach \x in {0, 1, 2, 3}
    \node[3] at (13-\x, 4) {$\x$};
\foreach \x in {1, 2, 3}
    \node[3] at (10-\x, 3) {$\x$};
\foreach \x in {1, 2, 3}
    \node[3] at (7-\x, 2) {$\x$};
\foreach \x in {1, 2, 3}
    \node[3] at (4-\x, 1) {$\x$};
\draw[style=forbidden line] (1,0) -- (12.7,3.9);

\foreach \x/\y/\z in { 0/0/1, 0/1/4, 0/2/7, 3/3/7, 6/4/7}
    \node[4] at (\x,\y) {\z};
   
\end{tikzpicture}
\vspace{0.5cm}
\end{minipage}

\caption{Correspondence between restricted permutations and lattice paths. On top: In the generating tree of $(122,123)$-avoiding permutations on regular multisets with $m=3$, the sequence $14777$ corresponds to the \pe\ $(443322421311)$. At the bottom: Labelled lattice for paths from $(0,0)$ to $(13,4)$ not touching the line $\Delta$ except at the endpoint. The points marked in gray cannot be reached by an up-step. The lattice path marked in black corresponds to the sequence $14777$.}
\label{fig.14777}
\end{figure}

Now, what kind of sequences can be obtained in this way? The first element of the sequence is always $1$, the second one always $(m+1)$ (these two elements could thus be omitted in the sequence since they bear no information). The third element can be one of the following: $(m+1), (m+2), \ldots, (2m+1)$. In general, if the $i$-th element in the sequence is $a_i$, the next element $a_{i+1}$ can be one of the following $(m+1), (m+2), \ldots, (a_i+m)$. 

The map $f:\mathcal{S}_{n,m}(122,123) \rightarrow \mathcal{B}_{n,m}$ where 
\[\mathcal{B}_{n,m} \coloneqq \left\lbrace (1,m+1,a_3 \ldots, a_n, a_{n+1}): a_i \in \mathbb{N} \wedge m+1 \leq a_i \leq a_{i-1}+m \right\rbrace\] described above is then a bijection. 

In order to establish a bijection between the mentioned lattice paths and $(122,123)$-avoiding permutations, we construct a bijection $g:\mathcal{P}_n(1,m) \rightarrow \mathcal{B}_{n,m}$. For this purpose we are going to label the points in the plane through which an allowed lattice path may lead. The label at a certain point corresponds to the number of different choices that can be made at this point: is one allowed to take an up-step and how many side-steps may one take? This number can easily  be calculated  with the help of the equation of $\Delta$. For any point $(x_i,y_i)$ with $y_i<n$ lying above the ``forbidden'' line the number of steps that can be made to the right is $m \cdot y_i -x_i$. Take note of the following for $y_i=n$: the line must be touched at the end which allows for one more side-step but no further up-steps may be taken. So in total, the number of possible steps at the point $(x_i,y_i)$ is equal to $(m \cdot y_i -x_i)+1$. For instance, in the example of lattice paths from $(0,0)$ to $(13,4)$, i.e., $m=3$ and $n=4$, shown in the bottom part of Figure~\ref{fig.14777} there are $5= (3\cdot 2-2)+1$ possibilities at the point $(2,2)$: one can take one step up; or one, two, three or four steps to the side. A fifth step to the side would touch the line and is therefore not allowed.

A lattice path from $(0,0)$ to $(1+mn,n)$ then defines a sequence of integers as shown in the following. We start at the point $(0,0)$ which has the label $1$, therefore the first element of our sequence is $1$. Then we walk along the lattice path and every time an up-step is taken, we note the label of the point that is reached by this up-step. This means that the second element of our sequence must always be $(m+1)$. For the example leading to the sequence $14777$, see again the bottom part of Figure~\ref{fig.14777}. 

The labels at level $i$, i.e., the labels of points with $y$-coordinate equal to $i$, range from $1$ to $(im+1)$ but not all these values can appear in the sequence. On the one hand, not all ``allowed'' points can be reached by up-steps. In Figure~\ref{fig.14777} these points are marked in gray. In this specific case, one notes that the points that cannot be reached by up-steps are those with labels $1, 2 \text{ or } 3=m$. 
For the general case, one can easily check that these points are those with labels between $1$ and $m$. On the other hand, one is only allowed to make right-steps and not left-steps and therefore all positions that lie to the left of the steps made so far can no longer be reached. This means that if the label at the $i$-th level is $a_i$, the positions labelled with $a_i+m+1, a_i+m+2, \ldots, m \cdot i +1$ cannot be reached at the next level. Thus the sequences obtained in the way described above have the following property: the first two elements are always $1$ and $(m+1)$, the $(i+1)$-th element $a_{i+1}$ ($i \geq 3$) can take any value between $(m+1)$ and $a_{i}+m$. Thus the obtained sequences are again in $\mathcal{B}_{n,m}$. It is clear that conversely any sequence in $\mathcal{B}_{n,m}$ can uniquely be identified with a lattice path in $\mathcal{P}_n(1,m)$, making the map $g$ a bijection and finishing this proof.
\end{proof}

\section{Avoiding the patterns 211 and 213}
\label{sec.122,231}

\begin{Thm} For all $m \geq 2$,  $s_{1,m}(211,213)=1$ and $s_{2,m}(211,213)=m+1$. 
For $n \geq 3$ it holds that
$s_{n,m}(211,213)=2 s_{n-1,m}(211,213) + s_{n-2,m}(211,213)$.
\label{Thm.211,213}
\end{Thm}

\begin{proof}
Suppose you are given a \pe\ $\sigma \in \mathcal{S}_{n,m}(211,213)$ and want to introduce $m$ copies of the element $(n+1)$ in order to generate a new \pe\ $\tilde{\sigma} \in [n+1]_m$ that avoids the patterns 211 and 213. 

First note that new elements may not be introduced before the $\left(m-1\right)$-th occurrence (from left to right) of the largest element otherwise a $211$-pattern will be created. We define 
\[
o_\sigma \coloneqq \text{min}\left\{i \in [nm]: \left|\left\{j \leq i: \sigma_j = n\right\}\right|=m-1 \right\}
\]
for $\sigma \in \mathcal{S}_{n,m}$. 

Let us denote by $d_\sigma$ the position of the first descent in $\sigma$, i.e., $d_\sigma= \text{min}\left\{i \in [n \cdot m]: \sigma_{i-1} > \sigma_{i}\right\}$.
Then no elements are allowed to be inserted to the right of the $d_\sigma$-th position in $\sigma$.
In case the permutation $\sigma \in [n]_m$ contains no descents at all, we set $d_\sigma=n\cdot m+1$.
The number of possible positions for the $(n+1)$-elements is then equal to $a_\sigma=d_\sigma-o_\sigma$.
For the empty permutation we set $d_\epsilon=1$ and $o_\epsilon=0$.
Note that if the difference $a_\sigma$ is negative, no new elements can be introduced without producing the forbidden patterns.
A difference $a_\sigma=0$ is never possible, since an occurrence of the largest element of a permutation can never be a descent.

We prove the following:
\begin{Cla} For $m \geq 2$ and $a_\sigma$ as defined above, it holds that:
\begin{compactitem}
\item $a_\sigma$ is always equal to $1$ or $2$ or is negative, 
\item if $a_\sigma$ is negative no new \pe\ is produced,
\item if $a_\sigma=1$ one new \pe\ $\tilde \sigma$ is produced with $a_{\tilde \sigma}=2$,
\item if $a_\sigma=2$ two new \pes $\tilde \sigma$ are produced with $a_{\tilde \sigma}=2$, one new one with $a_{\tilde \sigma}=1$ and $\left(m-2\right)$ new ones with $a_{\tilde \sigma}<0$.
\end{compactitem}
\label{claim2.211,213}
\end{Cla}

\begin{proof}[Proof of Claim \ref{claim2.211,213}]
We prove these statements by induction over $n$. For $n=0$ and $\sigma=\epsilon$ the first statement is true since, by definition, $a_\epsilon=1$. The empty permutation leads to one permutation with $n=1$, namely $\tilde \sigma = 11 \ldots 1$, for which it holds that  $a_{\tilde \sigma} = m+1-(m-1)=2$.

Now let $\sigma$ be a permutation in $\mathcal{S}_{n-1,m}$ with $o_\sigma=i$ and $n \geq 1$. By induction hypothesis $i$ can be negative or equal to 1 or 2.
In the first case clearly no new \pe\ can be produced since $a_\sigma < 0$ means that the first descent is to the left of the $\left(m-1\right)$-th occurrence of the largest element, thus every insertion of an $n$-element would produce a 211- or a 213-pattern. 

In the second case there is only a single possibility for placing all $m$ of the new elements, resulting in $o_{\tilde \sigma}=i+m-1$ and $d_{\tilde \sigma}=i+m+1$. Thus $a_{\tilde \sigma}=2$. 

Now, if the $n$-elements may be placed in two different positions, there are $m+1$ possibilities leading to $(211,213)$-avoiding \pes $\tilde \sigma$: place $0$ to $m$ elements in the first and the remaining elements in the second position. This leads to three different cases: placing $m$ elements in the first or $m$ elements in the second position leads to two new \pes with $a_{\tilde \sigma}=2$. Placing exactly one element in the second position leads to one new \pe\ with $a_{\tilde \sigma}=1$. If at least one element is placed in the first and at least two elements are placed in the second position (there are $\left(m-2\right)$ such possibilities), we have the following situation: $\tilde \sigma$ is of the form $\ldots \sigma_{i}\ldots n \sigma_{i+1} n \ldots n n \sigma_{i+2}$. Thus the first $21$-pattern is given by $n \sigma_{i+1}$, whereas the $(m-1)$-th occurrence of $n$ will always be to the right of $\sigma_{i+1}$ and therefore $a_{\tilde \sigma}<0$.

\end{proof}

We can sum up these results in the generating tree of $(211,213)$-avoiding permutations where the nodes are labelled by $a_\sigma$, as shown in Figure~\ref{tree211,213mlabels}. It is constructed by applying the following rewriting rule:
\begin{align}
\notag \text{Root: } & (1)  \\
\notag & (1)  \longrightarrow (2) \\
& (2)  \longrightarrow (2)(1)\underbrace{(N)\ldots (N)}_{\left(m-2\right)\text{times}}(2) 					\label{rule.211,213} \\
\notag & (N)  \longrightarrow \emptyset
\end{align}
\begin{figure}
\begin{center}
\begin{tikzpicture}
[1/.style={text centered, text=black},
2/.style={text centered, text=black},
N/.style={rectangle, fill=black, text centered, text=white, inner sep=2pt}, 
font=\tiny,
level 1/.style={level distance=0.8cm},
level 2/.style={sibling distance=1.2cm, level distance=1.3cm},
level 3/.style={sibling distance=.6cm, level distance=1.3cm},
level 4/.style={sibling distance=.3cm, level distance=1.3cm}]

\node (root) [1] {1} child { node [2] {2} 
child { node [2] {2}
		child{ node [2] {2}
				child{ node [2] {2}}
				child{ node [1] {1}}
				child{ node [N] (a){N}}
				child [missing]
				child{ node [N] (b){N}}
				child{ node [2] {2}}
		}
		child [missing]		
		child{ node [1] {1}
				child{ node [2] {2}}
		}
		child{ node [N] (c){N}}
		child{ node [N] (d){N}}		
		child{ node [2] {2}
				child{ node [2] {2}}
				child{ node [1] {1}}
				child{ node [N] (i){N}}
				child [missing]
				child{ node [N] (j){N}}
				child{ node [2] {2}}
		}
}
child [missing]
child [missing]
child { node [1] {1}
		child{ node [2] {2}
				child{ node [2] {2}}
				child{ node [1] {1}}
				child{ node [N] (k){N}}
				child [missing]
				child{ node [N] (l){N}}
				child{ node [2] {2}}
		}
}
child{ node [N] (e){N}
		child { node (g) {} edge from parent[draw=none] }
}
child{ node [N] (f){N}
		child { node (h) {} edge from parent[draw=none] }
}
child { node [2] {2}
	child{ node [2] {2}
			child{ node [2] {2}}
			child{ node [1] {1}}
			child{ node [N] (m){N}}
			child [missing]
			child{ node [N] (n){N}}
			child{ node [2] {2}}
	}
	child [missing]
	child{ node [1] {1}
			child{ node [2] {2}}
	}
	child{ node [N] (o){N}}
	child{ node [N] (p){N}}
	child{ node [2] {2}
			child{ node [2] {2}}
			child{ node [1] {1}}
			child{ node [N] (q){N}}
			child [missing]
			child{ node [N] (r){N}}
			child{ node [2] {2}}
	}
}
};
\path (a) -- (b) node [midway] {...};
\path (c) -- (d) node [midway] {...};
\path (e) -- (f) node [midway] {\ldots\ldots};
\path (i) -- (j) node [midway] {...};
\path (k) -- (l) node [midway] {...};
\path (m) -- (n) node [midway] {...};
\path (o) -- (p) node [midway] {...};
\path (q) -- (r) node [midway] {...};
\end{tikzpicture} 
\caption{The generating tree of $(211,213)$-avoiding permutations with nodes labelled by $a_\sigma$, the distance between the $\left(m-1\right)$-th occurrence of the largest element and the first descent. The black N-nodes, corresponding to \pes with negative $a_\sigma$ are repeated $\left(m-2\right)$ times in every ``block'' so that every 2-node has exactly $m+1$ children. N-nodes do not have any children.}
\label{tree211,213mlabels} 
\end{center}
\end{figure}
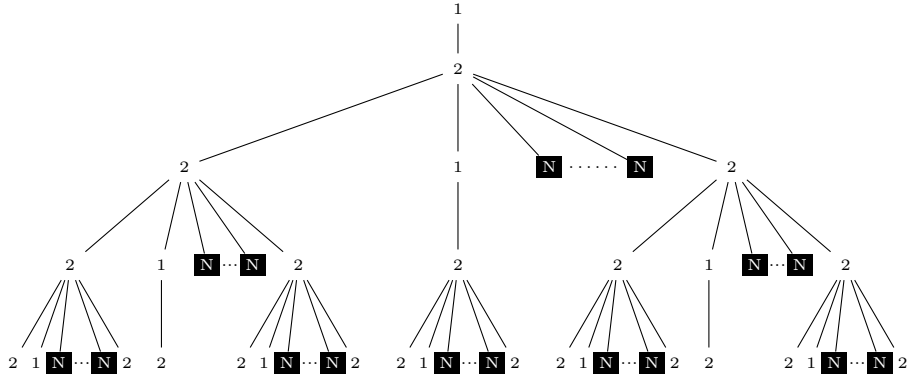

Let $s_m(n,i)$ be the number of nodes in the generating tree $\mathcal{T}_m$ labelled with $i$ at the height $n$. 
Translating the rewriting rule (\ref{rule.211,213}) into recurrence relations we obtain, for $n \geq 2$:
\begin{align}
\label{rec1}
s_m(n,N) & =  (m-2) \cdot s_m(n-1,2) = \left(m-2\right) \cdot s_m(n,1) \\ 
\label{rec2}
s_m(n,1) & =  s_m(n-1,2) \\ 
\label{rec3}
s_m(n,2) & =  s_m(n-1,1)  + 2 \cdot s_m(n-1,2) 
\end{align}
Now let $s_m(n) \coloneqq s_{n,m}(211,213)$ denote the number of permutations on $[n]_m$ avoiding the patterns 211 and 213.
Then recurrence~\eqref{rec1} leads to:
\[
s_m(n)=s_m(n,N)+s_m(n,1)+s_m(n,2)=\left(m-1\right) \cdot s_m(n,1) + s_m(n,2).
\]
Applying recurrences~\eqref{rec2} and \eqref{rec3} and using the above equation for $s_m(n)$ we obtain the following for $n \geq 3$:
\begin{align}
\notag s_m(n) 	& =  \left(m-1\right) \cdot s_m(n-1,2) + s_m(n,2) \\
\notag			& =  \left(m-1\right) \cdot s_m(n-2,1) + \underbrace{(m-1) \cdot 2 \cdot s_m(n-2,2)}_{=s_m(n-2,2)+(2m-3)\cdot s_m(n-2,2)} + s_m(n,2) \\
\notag 			& =  s_m(n-2)+(2m-3)\cdot s_m(n-1,1)+s_m(n-1,1)+2s_m(n-1,2)\\
\notag 			& =  s_m(n-2) +2(m-1) \cdot s_m(n-1,1) + 2s_m(n-1,2) \\
& =  s_m(n-2)+ 2s_m(n-1).
\label{rec.rel.211,213}
\end{align}
Following the rewriting rule (\ref{rule.211,213}), the initial values are given by: $s_m(1)=1$ and $s_m(2)=m+1$.
\end{proof}

\begin{Rem}
The sequence defined by the recurrence relation $s_m(n)=2s_m(n-1)+s_m(n)$ has a certain similarity to the well-known Fibonacci numbers $F_n=F_{n-1}+F_{n-2}$ with initial values $F_0=0$ and $F_1=1$. It is  a well-known fact that Fibonacci numbers can explicitly be computed using the so-called Binet-formula $F_n = \frac{\varphi^n-(1-\varphi)^n}{\sqrt 5} = \frac {\varphi^n-(-1/\varphi)^{n}}{\sqrt 5}$ where $\varphi = \frac{1 + \sqrt{5}}{2}$ is the golden ratio. We give a similar explicit formula for the sequence $s_{n,m}(211,213)$.
\label{rem.211,213.fibonacci}
\end{Rem}

\begin{Prop}
For $m \geq 2$ and $n \geq 1$, it holds that \vspace{-0.2cm}
\[
s_{n,m}(211,213) = \frac{1}{4} \left( \left(2-m\sqrt 2\right)\hspace{-0.1cm}\left(1 - \sqrt 2\right)^{n-1} \hspace{-0.3cm} + \left(2+m\sqrt 2\right)\hspace{-0.1cm}\left(1 + \sqrt 2\right)^{n-1}  													\right).
\]
\label{Thm.explicit.211,213}
\end{Prop}

\begin{proof}
Applying the recurrence relation \eqref{rec.rel.211,213} as well as the initial values of $s_{n,m}(211,213)$ to the generating function $A_m(z) = \sum_{n \geq 0} {s_{n+1,m}(211,213)z^n} $,  one obtains:
\begin{align}
A_m(z) = 2z\left(A_m(z)-1\right)+z^2A_m(z)+1+\left(m+1\right)z.
\label{gen.func.211,213}
\end{align}
Rearranging equation (\ref{gen.func.211,213}) and making use of partial fraction decomposition leads to
\begin{align*}
A_m(z)&=\frac{1+\left(m-1\right)z}{1-2z-z^2}\\
&=\frac{1}{4}\left(\frac{2 + 2\sqrt 2 -m\left(2 + \sqrt 2\right)}{z+1+\sqrt 2}+\frac{2 - 2\sqrt 2 - m\left(2- \sqrt 2\right)}{z+1-\sqrt 2}\right). \\
\end{align*}
This finally implies that 
\begin{align*}
s_{n,m}(211,213) &=  \left[z^{n-1}\right] A_m(z) \\
& =  \frac 14 \left[ \left(2-m\sqrt 2\right)\hspace{-0.1cm}\left(1 - \sqrt 2\right)^{n-1} \hspace{-0.2cm} + \left(2+m\sqrt 2\right)\hspace{-0.1cm}\left(1 + \sqrt 2\right)^{n-1} \right]
\end{align*}
 for all $ n \geq 1$.
\end{proof}

\section{Avoiding the patterns 122 and 213}
\label{sec.122,213}

\begin{Thm}
For $m \geq 2$ we have $s_{1,m}(122,213)=1$ and $s_{2,m}(122,213)=m+1$.
For all $n \geq 3$ it holds that $s_{n,m}(122,213)=m s_{n-1,m}(122,213) + s_{n-2,m}(122,213)$.
\label{Thm.122,213}
\end{Thm}

\begin{figure}
\begin{center}
\begin{tikzpicture}
[1/.style={ text centered},
2/.style={rectangle, draw,  text centered, inner sep =2pt},
eps/.style={ text centered}, 
font=\small,
level 2/.style={sibling distance=2.5cm},
level 3/.style={sibling distance=0.7cm},
level distance=1.3cm]

\node (root) [eps] {1} child { node [2] {m+1}
child { node [2] {m+1}
		child{ node [2] {m+1}}
		child{ node [1] {m}}
		child{ node [1] (a){m}}
		child{ node [1] (b){m}}
}
child { node [1] {m}
		child{ node [2] {m+1}}
		child{ node [1] (c){m}}
		child{ node [1] (d){m}}
}
child { node [1] (e) {m}
		child{ node [2] {m+1}}
		child{ node [1] (g){m}}
		child{ node [1] (h){m}}
}
child [missing]
child { node [1] (f) {m}
		child{ node [2] {m+1}}
		child{ node [1] (i){m}}
		child{ node [1] (j){m}}
}
};
\path (a) -- (b) node [midway] {...};
\path (c) -- (d) node [midway] {...};
\path (e) -- (f) node [midway] {\ldots\ldots};
\path (g) -- (h) node [midway] {...};
\path (i) -- (j) node [midway] {...};
\end{tikzpicture}  
\caption{Generating tree of $(122,213)$-avoiding permutations with nodes labelled by $d_\sigma$, the position of the first descent. The dots between two $(m+1)$-nodes represent to $\left(m-3\right)$  $m$-nodes, so that a $\left(m+1\right)$-node has exactly $\left(m+1\right)$ children and a $m$-node has exactly $m$ children.}
\label{tree122,213mlabels}
\end{center} 
\end{figure}
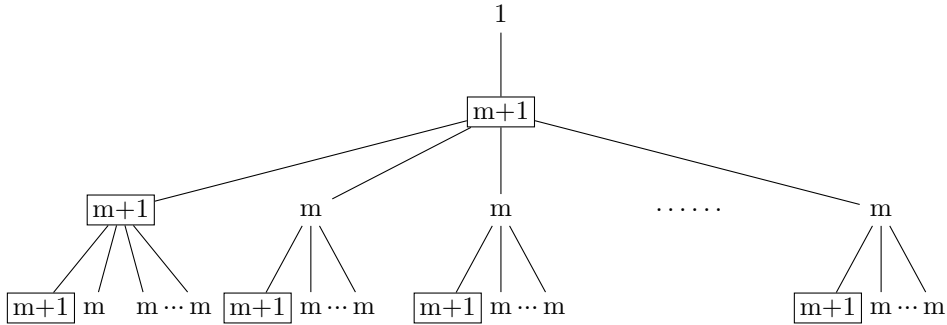

\begin{proof}
Let $\sigma$ be a permutation in $S_{n,m}(122,213)$. 
In order to produce an element $\tilde{\sigma}$ in $S_{n+1,m}(122,213)$ by introducing $m$ copies of the element $(n+1)$, $\left(m-1\right)$ copies of the element $(n+1)$ have to be placed at the beginning of $\sigma$ and the remaining $\left(n+1\right)$ has to be placed somewhere to the left of the position $d_\sigma$, which denotes the position of the first descent in $\sigma$ as defined in the proof of Theorem~\ref{Thm.211,213}.

Similar arguments as those used in the proof of Theorem \ref{Thm.211,213} lead to the following rewriting rule for the generating tree of $(122,213)$-avoiding permutations (as depicted in Figure~\ref{tree122,213mlabels}) with nodes labelled by $d_\sigma$:
\begin{align}\label{rew.rule.122,213}
\notag \text{Root: } & (1)  \\ 
\notag & (1)  \longrightarrow (m+1) \\
& (m+1)  \longrightarrow (m+1)\underbrace{(m)\ldots (m)}_{m\text{ times}} \\
\notag & (m)  \longrightarrow (m+1)\underbrace{(m)\ldots (m)}_{\left(m-1\right)\text{ times}}
\end{align}
Let $s_m(n,i)$ be the number of nodes in the generating tree labelled with $i$ at height $n$. Then $s_m(n)=s_m(n,m)+s_m(n,m+1)$ is the number  of permutations on $[n]_m$ avoiding the patterns 122 and 213 for $n\geq 1$.
Manipulating the recurrence relations given by the rewriting rule, one obtains
$s_m(n) =  ms_m(n-1) + s_m(n-2) \text{ for } n \geq 3$
and the initial values $s_m(1)=1$, $s_m(2)=m+1$.
\end{proof}

\begin{Rem}
This sequence defined by the recurrence relation $s_n=ms_{n-1}+s_{n-2}$ can again be seen as a generalization of Fibonacci numbers, cf. Remark \ref{rem.211,213.fibonacci}. An explicit formula can also be given in this case.
\end{Rem}

\begin{Prop}
\label{coeff122,213}
For $m \geq 2$ and $n \geq 1$, it holds that
\begin{align*}
S_{n,m}(211,213) &= \frac{2^{-n}}{\sqrt{m^2+4}} \bigg( \left(2 + \sqrt{m^2+4} +m\right)\left(m+\sqrt{m^2+4}\right)^{n-1}  \\
&\textcolor{white}{=}    - \left(2 - \sqrt{m^2+4}+m\right) \left(m-\sqrt{m^2+4}\right)^{n-1}\bigg).
\end{align*}
\end{Prop}

\section{Avoiding the patterns 122 and 312}
\label{sec.122,312}

\begin{Thm}
$s_{n,m}(122,312)=\left(n-1\right)\cdot m +1$ for all $n \in \mathbb{N}, m \geq 2$ .
\label{Thm.122,312}
\end{Thm}

\begin{proof}
One can easily see that a permutation $\sigma$ on $[n]_m$ that avoids both the patterns 122 and 312 must be of the following form 
\[
\sigma = \underbrace{n \ldots n}_{(m-1) \text{ times}} \tau,
\]
where $\tau$ is constructed by inserting the element $n$ at any position of 
\[
\underbrace{(n-1) \ldots (n-1)}_{m \text{ times}}\underbrace{(n-2) \ldots (n-2)}_{m \text{ times}} \ldots  \underbrace{1 \ldots 1}_{m \text{ times}}.
\]

On the one hand it is clear that at least $\left(m-1\right)$ $n$'s have to stand at the beginning of $\sigma$ if a 122-pattern should be avoided. 
On the other hand, if the copies of the elements in $[n-1]$ are not arranged in decreasing order, i.e., if $\tau$ contained a $12$-pattern, $\sigma$ would certainly contain a $312$ pattern. 
Since the length of $\tau$ is $(n-1)m$, there are $(n-1)m + 1$ possibilities of inserting the element $n$ into $\tau$ and the total number of permutations in $\mathcal{S}_{n,m}(122,312)$ is $(n-1)m + 1$.
\end{proof}

\section{Avoiding the patterns 122 and 321}
\label{sec.122,321}

\begin{Thm}
For $m\geq2$ 
\[
s_{n,m}(122,321)=\begin{cases} 1 & \text{for } n=1, \\ m+1 & \text{for } n=2, \\ 0 & \text{for all } n \geq 3.\end {cases}
\]
\label{Thm.122,321}
\end{Thm}

\begin{proof}
For $n\leq2$, all permutations on $[n]_m$ avoid the pattern $321$ and thus $\mathcal{S}_{n,m}(122,321)=\mathcal{S}_{n,m}(122)$. Clearly, $s_{1,m}(122,321)=1$ and $s_{2,m}(122,321)=m+1$.

As seen for $122$-avoiding permutations in the proof of Theorem~\ref{Thm122,123=122,132}, the elements $n^{m-1}, \left(n-1\right)^{m-1}, \ldots,$ $2^{m-1}, 1^{m-1}$ always appear in this order from left to right (and are left-to-right-minima). Thus, for $n\geq3$ and $m\geq2$, a $n\left(n-1\right)\left(n-2\right)$-subsequence can always be found and every $\sigma \in \mathcal{S}_{n,m}$ contains the pattern $321$.
\end{proof}

\section{The Stanley-Wilf conjecture for permutations on multisets}
\label{sec:StanleyWilf}

When taking a look at the enumerative results obtained for permutations on regular multisets avoiding a pair of patterns of length three, summed up in Table \ref{tab:all.results}, one can observe the following:
All formul\ae\ in the first (Pairs of ordinary patterns) and third (Pairs of one ordinary and one multiset-pattern) column of Table \ref{tab:all.results} can be bounded by $c^{n \cdot m}$ for some constant $c$. In the second column (Pairs of multiset-patterns) this is not the case since $s_{n,m}(212, 121)=n!$. 

This suggests the following conjecture: a generalized version of the Stanley-Wilf conjecture also holds for permutations on (regular) multisets that avoid an ordinary pattern, but does not hold for permutations avoiding a multiset pattern.
In the following we prove this conjecture and show how it may be formulated for patterns in words.

Let us first recall the original version of the former Stanley-Wilf conjecture which states that the number of $n$-\pe s avoiding an arbitrary given pattern does not grow faster than exponentially:

\begin{Thm}
[Stanley-Wilf Conjecture, 1990, proven 2004 in \cite{klazar2000füredi,marcus2004excluded}] Let $\pi$ be an arbitrary pattern and $s_n(\pi)$ denote the number of permutations of $[n]$ that avoid the pattern $\pi$. 
Then there exists a constant $c_\pi$ such that for all positive integers it holds that 
\begin{align}
s_n(\pi) \leq c_\pi^n.
\end{align}
\label{Conj_Stanley_Wilf}
\end{Thm}

A generalization of the Stanley-Wilf conjecture has been considered independently by Klazar and Marcus \cite{klazar2007extensions}  and by Balogh, Bollob\'as and Morris \cite{DBLP:journals/ejc/BaloghBM06}. Klazar and Marcus proved an exponential bound on the number of hypergraphs avoiding a fixed permutation, settling various conjectures of Klazar as well as a conjecture of Br\"and\'en and Mansour. Balogh, Bollob\'as and Morris went even further in their generalization and showed similar results for the growth of hereditary properties of partitions, ordered graphs and ordered hypergraphs. For details, please consider the original work. The results in \cite{DBLP:journals/ejc/BaloghBM06,klazar2007extensions} being very general however required rather involved proofs. In both papers, a generalized version of the F\"uredi-Hajnal conjecture (Theorem \ref{Conj_Fueredi_Hajnal}) was first formulated and the proof of several intermediary results was necessary.

From the results in both these papers it follows that the Stanley-Wilf conjecture also holds for permutations on multisets respectively words since these can be represented with the help of bipartite graphs which are (very) special cases of hypergraphs. 
To the best of our knowledge this fact has however not yet been stated in the literature. 
We therefore wish to stress this point here, as pattern avoidance in permutations on multisets and words has attracted a great deal of interest over the past few years. 

In the following we show that Klazar's proof \cite{klazar2000füredi} can very easily be formulated for words or for permutations on multisets without requiring the employment of any other results or a further generalization of already known results. 
Let us first recapitulate how the Stanley-Wilf conjecture was proven.

Stanley-Wilf was not proven directly but via another conjecture concerning pattern avoidance in binary matrices formulated by F\"uredi and Hajnal in~\cite{furedi1992davenport}. 

\begin{Def}
Let $P$ and $Q$ be matrices with entries in $\left\lbrace 0,1 \right\rbrace $ and let $Q$ have dimension $m\times n$. We say that the matrix $P$ \emph{contains} $Q$ as a pattern, if there is a submatrix $\tilde Q$ of $P$, so that $\tilde{Q}_{i,j}=1$ whenever $Q_{i,j}=1$ for $i \leq m$ and $j \leq n$. If there is no such submatrix $\tilde Q$, we say that $P$ \emph{avoids} $Q$.
\label{def_avoidance_matrices}
\end{Def}

\begin{Thm}[F\"uredi-Hajnal Conjecture]
Let $Q$ be any \pe\ matrix. We define $f(n,Q)$ as the maximal number of $1$-entries that a $Q$-avoiding $(n \times n)$-matrix $P$ may have. Then there exists a constant $d_Q$ so that
\[
f(n,Q) \leq d_Q \cdot n.
\]
\label{Conj_Fueredi_Hajnal}
\end{Thm} 

In the year 2000 Martin Klazar proved that the F\"uredi-Hajnal conjecture implies the Stanley-Wilf conjecture \cite{klazar2000füredi}. Four years later the F\"uredi-Hajnal conjecture was proven by Adam Marcus and G\'abor Tardos in \cite{marcus2004excluded}, finally providing a  proof of the long-standing Stanley-Wilf conjecture. 
We refer the reader to the original literature for the proof of the F\"uredi-Hajnal conjecture.

For the proof of the generalization of the Stanley-Wilf conjecture to permutations on multisets respectively to words, it will merely be necessary to show that F\"uredi-Hajnal implies a generalized version of Stanley-Wilf. Let us therefore briefly recall the argument used by Martin Klazar in his proof. 

In order to establish a connection between pattern avoidance in matrices and pattern avoidance in \pe s Klazar takes an elegant detour via pattern avoidance in simple bipartite graphs and defines the following notion of pattern containment:

\begin{Def}
Let $P([n],[n'])$ and $Q([k],[k'])$ be simple bipartite graphs, where $k \leq n$ and $k' \leq n'$. Then we say that $P$ \emph{contains $Q$ as an ordered subgraph} if two order preserving injections $f:[k]\rightarrow[n]$ and $f':[k']\rightarrow[n']$ can be found so that if $vv'$ is an edge of $Q$, then $f(v)f'(v')$ is an edge of $P$.
\label{Def:graph.avoid}
\end{Def}

Clearly, every \pe\ can be identified with a simple bipartite graph in a unique way. For a \pe\ $\sigma$ on $[n]$ the associated graph $G_\sigma$ is the bipartite graph with vertex set $([n],[n])$ and where $e=(i,j)$ is an edge iff $\sigma_i=j$ in $\sigma$. 
Then the following is a direct consequence: If the \pe\ $\sigma$ contains a \pe\ $\pi$ as a pattern, then $G_\sigma$ contains $G_\pi$ as an ordered subgraph. Reversely, if $\sigma$ avoids $\pi$, $G_\sigma$ will also avoid $G_\pi$. 
However, not every simple bipartite graph corresponds to a \pe\ (this is only the case if the vertex degree is equal to one for all vertices), thus $s_n(\pi) \leq  g_n(\pi)$, where $g_n(\pi)$ is the number of simple bipartite graphs on $([n],[n])$ avoiding the graph $G_\pi$ corresponding to a \pe\ $\pi$. 

Let $G$ be a simple bipartite graph on $([n],[n])$ that avoids $G_\pi$.
Then the adjacency matrix $A(G)$ of $G$ avoids the adjacency matrix $A(G_\pi)$ and Theorem \ref{Conj_Fueredi_Hajnal} implies that $A(G)$ can have at most $d_\pi \cdot n=d_{A(G_{\pi})} \cdot n$ entries equal to $1$, respectively that $G_\pi$ can have at most $d_\pi \cdot n$ edges.
By gradually contracting the graph $G$ - reducing its size to half in every step without loosing the $G_\pi$-avoiding property - Klazar shows that this leaves at most an exponential number of possibilities for the graph $G$: $g_n(\pi) \leq 15^{2d_\pi n}$. Thus $s_n(\pi) \leq c_\pi ^n$ with $c_\pi = 15^{2d_\pi}$.

Our goal is to prove the following generalization of Theorem \ref{Conj_Stanley_Wilf}:
\begin{Cor}
Let $\pi$ be an arbitrary \pe\ on an \emph{ordinary} set.
If $s_\mu(\pi)=s_{\mu(1), \ldots, \mu(n)}(\pi)$ denotes the number of \pe s on the \mul\ $\left\lbrace 1^{\mu(1)}, \ldots, n^{\mu(n)}\right\rbrace $ with $\sum_{i=1}^n \mu(i)=l$ avoiding $\pi$ and $w_{l,n}(\pi)$ denotes the number of words of length $l$ over the alphabet $[n]$ avoiding $\pi$ where $l\geq n$, there exists a constant $e_\pi$ so that the following holds:
\begin{align}
s_\mu(\pi) \leq e_\pi^{l} \text{ and } w_{l,n}(\pi) \leq e_\pi^{l}.
\label{eqn_multi_sw}
\end{align}
\label{Multi_Stanley_Wilf}
\end{Cor}

\begin{proof}
We merely need to show that $s_\mu(\pi)=s_{\mu(1), \ldots, \mu(n)}(\pi) \leq  g_{l}(\pi)$ and that  $w_{l,n}(\pi) \leq  g_{l}(\pi)$ in case $l\geq n$.
This is not difficult: For the first case, i.e., the number of permutations on the multiset $\left\lbrace 1^{\mu(1)}, \ldots, n^{\mu(n)}\right\rbrace$ avoiding the pattern $\pi$, observe that every element $\sigma \in \mathcal{S}_\mu$ can represented by a simple bipartite graph $G_\sigma$ on the vertex set $([l],[n])$ where $l=\sum_{i=1}^n \mu(i)$ is the length of $\sigma$. 
If $\sigma$ avoids the permutation pattern $\pi$, then $G_\sigma$ avoids the graph $G_\pi$ in the sense defined in Definition~\ref{Def:graph.avoid}.

For the second case, i.e., for words of length $l \geq n$ we proceed similarly.
A word of length $l$ over an alphabet of size $n$ that avoids the permutation pattern $\pi$ can be represented by a simple bipartite graph $G_\sigma$ on the vertex set $([l],[n])$ avoiding the graph $G_\pi$.

For both cases we can add $(l-n)$ (which is non-negative in both cases) vertices to the second vertex set of $G_\sigma$, i.e., to $[n]$, without introducing any new edges and obtain a balanced simple bipartite graph $\tilde{G}_\sigma$ on the vertex set $([l],[l])$ that avoids $G_\pi$.
We thus obtain that the following inequalities hold: $s_\mu(\pi) \leq g_l(\pi)$ and $w_{l,n}(\pi) \leq  g_{l}(\pi)$.
It is known from Klazar's proof that $g_l(\pi) \leq 15^{2d_\pi l}$. 
This proves both the first and the second statement of \eqref{eqn_multi_sw} with $e_\pi=15^{2d_\pi}$.
\end{proof}

\begin{Rem}
For the case that $l<n$ we can show that $w_{l,n}(\pi) \leq e_\pi^{n}$ which however is in general no improvement over $n^l$, i.e., the total number of words of length $l$ over the alphabet $[n]$.

It this case we can however not hope to obtain an exponential bound of the type $e_\pi^{l}$ as can be seen by taking a look at the following simple example.
Let $\pi$ be the pattern $12$, thus the words of length $l$ over $[n]$ avoiding $\pi$ are those containing no ascents.
Let us restrict ourselves to permutations of length $l$ over a subset of the alphabet $[n]$.
Given a choice of $l$ letters from $[n]$, there is only a single permutation that avoids $\pi$  namely the decreasing one.
Thus there are exactly $\binom{n}{l}$ permutations of length $l$ and consisting of letters from $[n]$ that avoid $\pi$.
Clearly $\binom{n}{l} \geq (n/l)^l$ and therefore $w_{l,n}(12) \geq (n/l)^l$ which can never be bounded by $e_\pi^{l}$ for arbitrary $n$.
Thus in general, a bound of the type $e_\pi^{l}$ is not possible for $w_{l,n}(\pi)$ in the case that the alphabet is larger than the word is long.
\label{Rem:l<n}
\end{Rem}

\begin{Rem}
The condition that the avoided pattern $\pi$ is a \pe\ on an ordinary set is crucial in Corollary~\ref{Multi_Stanley_Wilf}. 
Indeed, the number $s_{n,m}(\pi)$ may grow faster than exponentially if the pattern $\pi$ is a \mul -\pe . For example, consider $m$-Stirling \pe s, i.e., \pe s on a regular \mul\ with multiplicity $m$ avoiding the pattern $212$. As stated in \cite{kuba2010enumeration}, the number of $m$-Stirling \pe s on the set $[n]$ is equal to $n!m^n \binom {n-1+1/m}{n}$ and thus grows super-exponentially.

The same remark holds for the number of words avoiding a given pattern.
\end{Rem}

\bibliographystyle{abbrv}
\bibliography{references}

\begin{thebibliography}{10}

\bibitem{albert2001permutations}
M.~Albert, R.~Aldred, M.~Atkinson, C.~Handley, and D.~Holton.
\newblock Permutations of a multiset avoiding permutations of length 3.
\newblock {\em European Journal of Combinatorics}, 22(8):1021--1031, 2001.

\bibitem{andré1887solution}
D.~Andr{\'e}.
\newblock Solution directe du probleme r{\'e}solu par {M}. {B}ertrand.
\newblock {\em CR Acad. Sci. Paris}, 105:436--437, 1887.

\bibitem{DBLP:journals/ejc/BaloghBM06}
J.~Balogh, B.~Bollob{\'a}s, and R.~Morris.
\newblock Hereditary properties of partitions, ordered graphs and ordered
  hypergraphs.
\newblock {\em European Journal of Combinatorics}, 27(8):1263--1281, 2006.

\bibitem{banderier2002generating}
C.~Banderier, M.~Bousquet-M\'{e}lou, A.~Denise, P.~Flajolet, D.~Gardy, and
  D.~Gouyou-Beauchamps.
\newblock Generating functions for generating trees.
\newblock {\em Discrete mathematics}, 246(1-3):29--55, 2002.

\bibitem{bona_combinatorics_2004}
M.~B\'{o}na.
\newblock {\em Combinatorics of permutations}.
\newblock Discrete Mathematics and Its Applications. Chapman \& {Hall/CRC},
  2004.

\bibitem{DBLP:journals/combinatorics/Bousquet-Melou02}
M.~Bousquet-M{\'e}lou.
\newblock Four classes of pattern-avoiding permutations under one roof:
  Generating trees with two labels.
\newblock {\em Electr. J. Comb.}, on(2), 2002.

\bibitem{DBLP:journals/dm/Bousquet-MelouP00}
M.~Bousquet-M{\'e}lou and M.~Petkov\v{s}ek.
\newblock Linear recurrences with constant coefficients: the multivariate case.
\newblock {\em Discrete Mathematics}, 225(1-3):51--75, 2000.

\bibitem{furedi1992davenport}
Z.~F\"uredi and P.~Hajnal.
\newblock {Davenport-Schinzel theory of matrices}.
\newblock {\em Discrete Mathematics}, 103(3):233--251, 1992.

\bibitem{heubach2006avoiding}
S.~Heubach and T.~Mansour.
\newblock Avoiding patterns of length three in compositions and multiset
  permutations.
\newblock {\em Advances in Applied Mathematics}, 36(2):156--174, 2006.

\bibitem{heubach2009combinatorics}
S.~Heubach and T.~Mansour.
\newblock {\em Combinatorics of compositions and words}.
\newblock Chapman \& Hall/CRC, 2009.

\bibitem{kitaev2011patterns}
S.~Kitaev.
\newblock {\em Patterns in Permutations and Words}.
\newblock Springer-Verlag New York Inc, 2011.

\bibitem{klazar2000füredi}
M.~Klazar.
\newblock {The F\"uredi-Hajnal conjecture implies the Stanley-Wilf conjecture}.
\newblock {\em Formal Power Series and Algebraic Combinatorics}, pages
  250--255, 2000.

\bibitem{klazar2007extensions}
M.~Klazar and A.~Marcus.
\newblock Extensions of the linear bound in the {F}{\"u}redi-{H}ajnal
  conjecture.
\newblock {\em Advances in Applied Mathematics}, 38(2):258--266, 2007.

\bibitem{knuth1968art}
D.~Knuth.
\newblock {\em The art of computer programming. Vol. 1: Fundamental
  algorithms}.
\newblock Addison-Wesley Series in Computer Science and Information Processing,
  1968.

\bibitem{kuba2010enumeration}
M.~Kuba and A.~Panholzer.
\newblock {Enumeration formul\ae\ for pattern restricted Stirling
  permutations}.
\newblock {\em preprint}, 2010.

\bibitem{marcus2004excluded}
A.~Marcus and G.~Tardos.
\newblock {Excluded permutation matrices and the Stanley-Wilf conjecture}.
\newblock {\em Journal of Combinatorial Theory, Series A}, 107(1):153--160,
  2004.

\bibitem{mohanty1979lattice}
G.~Mohanty.
\newblock {\em Lattice path counting and applications}.
\newblock Academic Press New York, 1979.

\bibitem{myers2007pattern}
A.~Myers.
\newblock Pattern avoidance in multiset permutations: bijective proof.
\newblock {\em Annals of Combinatorics}, 11(3):507--517, 2007.

\bibitem{rothe1793formulae}
H.~Rothe.
\newblock {\em Formulae De Serierum Reversione Demonstratio Universalis Signis
  Localibus Combinatorio-Analyticorum Vicariis Exhibita: Dissertatio
  Academica}.
\newblock Sommer, 1793.

\bibitem{savage2006pattern}
C.~Savage and H.~Wilf.
\newblock Pattern avoidance in compositions and multiset permutations.
\newblock {\em Advances in Applied Mathematics}, 36(2):194--201, 2006.

\bibitem{simion1985restricted}
R.~Simion and F.~Schmidt.
\newblock Restricted permutations.
\newblock {\em European Journal of Combinatorics}, 6(4):383--406, 1985.

\bibitem{west1996generating}
J.~West.
\newblock Generating trees and forbidden subsequences.
\newblock {\em Discrete Mathematics}, 157(1):363--374, 1996.

\end{thebibliography}
\label{lastpage}

\end{document}